\documentclass[a4paper]{article}
\usepackage{amsthm,amsmath,amssymb}
\usepackage[OT1]{fontenc}
\usepackage{graphicx}
\usepackage[utf8]{inputenc}
\usepackage{enumitem}

\usepackage{tabularx,booktabs}

\usepackage{authblk}
\usepackage[numbers,sort]{natbib}

\usepackage[usenames,dvipsnames]{xcolor}

\newtheorem{theorem}{Theorem}
\newtheorem{lemma}[theorem]{Lemma}
\newtheorem{proposition}[theorem]{Proposition}

\newtheorem{corollary}[theorem]{Corollary}

\newtheorem*{claim}{Claim}

\theoremstyle{remark}
\newtheorem*{remark}{Remark}

\newenvironment{claimproof}{%
\noindent%
\textit{Proof of Claim.}%
}
{
\hfill $\triangle$%
\medskip
}

\let\leq\leqslant
\let\geq\geqslant
\let\setminus\smallsetminus

\newcommand{\cA}{\mathcal{A}}

\newcommand{\cF}{\mathcal{F}}
\newcommand{\cC}{\mathcal{C}}
\newcommand{\TD}{\mathcal{TD}}

\newcommand{\bN}{\mathbb{N}}

\newcommand{\tw}{\operatorname{tw}}
\newcommand{\td}{\operatorname{td}}
\newcommand{\dis}{\operatorname{dis}}

\makeatletter
\let\old@setaddresses\@setaddresses
\def\@setaddresses{\bgroup\parindent 0pt\let\scshape\relax\old@setaddresses\egroup}
\makeatother

\DeclareGraphicsExtensions{.pdf,.png,.eps}
\graphicspath{{figs/}}

\begin{document}

\title{The $k$-strong induced arboricity of a graph}

\author{Maria Axenovich, Daniel Gon\c{c}alves, Jonathan Rollin, Torsten Ueckerdt}

\maketitle
 
\begin{abstract}
 The induced arboricity of a graph $G$ is the smallest number of induced forests covering the edges of $G$.
 This is a well-defined parameter bounded from above by the number of edges of $G$ when each forest in a cover consists of exactly one edge.
 Not all edges of a graph necessarily belong to induced forests with larger components. 
 For $k\geq 1$, we call an edge $k$-valid if it is contained in an induced tree on $k$ edges.
 The $k$-strong induced arboricity of $G$, denoted by $f_k(G)$, is the smallest number of induced forests with components of sizes at least $k$ that cover all $k$-valid edges in $G$. 
 This parameter is highly non-monotone.
 
 However, we prove that for any proper minor-closed graph class $\cC$, and more generally for any class of bounded expansion, and any $k \geq 1$, the maximum value of $f_k(G)$ for $G \in \cC$ is bounded from above by a constant depending only on $\cC$ and $k$.
 This implies that the adjacent closed vertex-distinguishing number of graphs from a class of bounded expansion is bounded by a constant depending only on the class.
 We further prove that $f_2(G) \leq 3\binom{t+1}{3}$ for any graph $G$ of tree-width~$t$ and that $f_k(G) \leq (2k)^d$ for any graph of tree-depth $d$. 
 In addition, we prove that $f_2(G) \leq 310$ when $G$ is planar.
\end{abstract}
%

\section{Introduction}\label{sec:introduction}

Let $G = (V,E)$ be a simple, finite and undirected graph.
An \emph{induced forest} in $G$ is an acyclic induced subgraph of $G$.
A \emph{cover} of $X\subseteq V\cup E$, is a set of subgraphs of $G$ whose union contains every element of $X$.
It is certainly one of the most classical problems in graph theory to cover the vertex set $V$ or the edge set $E$ of $G$ with as few as possible subgraphs from a specific class, such as independent sets~\cite{West}, stars~\cite{AK85}, paths~\cite{AEH81}, forests~\cite{Nas64}, planar graphs~\cite{MOS98}, interval graphs~\cite{GW95}, or graphs of tree-width~$t$~\cite{DOS00}, just to name a few.
Extensive research on graph covers has been devoted to the following two graph parameters:
The \emph{vertex arboricity} (also called point-arboricity) of $G$, denoted by $a(G)$, is the minimum $t$ such that~$V$ can be covered with $t$ induced forests~\cite{CKW68}.
Note that $V$ can always be covered using a single (not necessarily induced) forest.
The \emph{edge arboricity} of $G$, denoted as $a'(G)$, is the minimum~$t$ such that $E$ can be covered with $t$ forests~\cite{Nas64}.
Burr~\cite{Bur86} proved that for any graph $G$ we have $a(G) \leq a'(G)$.
Nash-Williams~\cite{Nas64} proved that the edge arboricity of a graph $G$ is given by $a'(G) = \max\{\lceil|E(H)|/(|V(H)|-1)\rceil\}$ where the maximum is taken over all subgraphs $H$ of $G$ with at least two vertices.
Thus the problem of covering the edge set $E$ with forests is completely answered as it depends only the maximum density among subgraphs of $G$.
However, if forests are required to be induced, the graph's structure plays a more important role.

Here, we define the {\bf \emph{induced arboricity}} $f_1(G)$ of $G$ to be the minimum $t$ such that the edge set $E$ of $G$ can be covered with $t$ induced forests.
This means that not only all components of such a forest $F$ are induced trees, but that $F$ is an induced subgraph of $G$, i.e., also no edge of $G$ connects two trees in $F$.
The induced arboricity $f_1(G)$ is a natural arboricity-parameter.
While for $a(G)$ one covers the vertices of $G$ with induced forests, and for $a'(G)$ one covers the edges of $G$ with general forests, for $f_1(G)$ one covers the edges of $G$ with induced forests, which to the best of our knowledge has not been considered before.
As it is trivial to cover the vertices with a single (not necessarily induced) forest, our study completes all four cases of covering the vertices or edges with induced or general forests.

Recent results on vertex-distinguishing numbers show connections to edge coverings with induced forests in which each component has at least two edges.
In~\cite{AHP16} it is shown that the adjacent closed vertex-distinguishing number of $G$, denoted by $\dis[G]$, is bounded from above by a function of the smallest number $m$ of induced forests with components on at least two edges covering the edges of $G$.
We provide the precise definition of this variant of distinguishing numbers in Section~\ref{sec:preliminaries}.
Moreover, it is asked in~\cite{AHP16} whether there exists a constant $m$ such that for any planar graph such a cover exists.
Here we answer this question in the affirmative by proving that $m = 310$ is enough (c.f.\ Theorem~\ref{thm:acyclic}).

Motivated by this connection to vertex-distinguishing numbers of graphs, we study the more general problem of covering the edges of $G$ with induced forests in which each component has at least $k$ edges.
More precisely, for $k \geq 1$, let a \emph{$k$-strong forest} of $G$ be an induced forest in $G$, each of whose connected components consists of at least $k$ edges.
An edge $e \in E(G)$ is defined to be \emph{$k$-valid}, $k \geq 1$, if there exists a $k$-strong forest in $G$ containing~$e$.
We define the {\bf \emph{$k$-strong induced arboricity}} of $G$, denoted by $f_k(G)$, as the smallest number of $k$-strong forests covering all $k$-valid edges of $G$.

\begin{figure}[tb]
 \centering
 \includegraphics{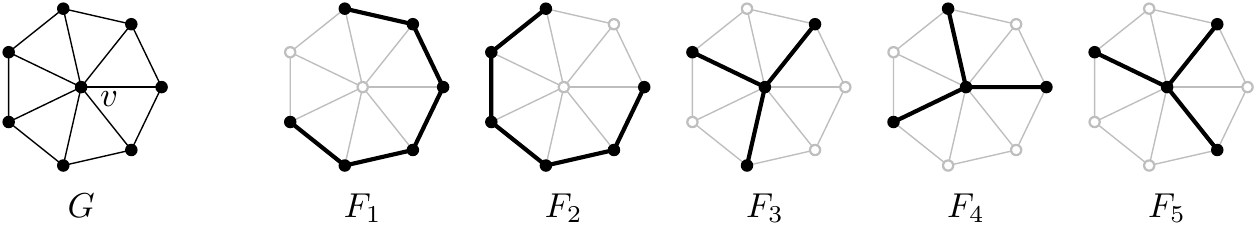}
 \caption{A graph $G$ and five induced forests $F_1,\ldots,F_5$ in $G$.
  Edges incident to $v$ are $3$-valid but not $4$-valid and the remaining edges are $5$-valid but not $6$-valid.
  }
 \label{fig:small-examples}
\end{figure}

For example, a $1$-strong forest is a forest that is induced and has no isolated vertices, and a $2$-strong forest is one that is induced and has neither $K_1$- nor $K_2$-components.
Here a $K_1$-component, respectively $K_2$-component, of a forest $F$ is a connected component of~$F$ with exactly one vertex, respectively exactly two vertices.
Note that induced arboricity of $G$ and $1$-strong induced arboricity of $G$ coincide, justifying the notation $f_1(G)$, because isolated vertices in a forest do not help to cover the edges of $G$ and hence these can be easily omitted.
Thus for the induced arboricity it suffices to consider induced forests where every component has at least one edge, that is, $1$-strong forests.
However, note that for $k \geq 2$ we possibly can not cover $E(G)$ with $k$-strong forests, for example when $G$ is a clique or when $|E(G)| < k$.
In fact, only the $k$-valid edges of $G$ can be covered with $k$-strong forests, where of course, every edge is $1$-valid.
By removing a leaf in an induced tree, one obtains an induced tree with exactly one edge less.
Thus an edge $e$ is $k$-valid if and only if it belongs to an induced tree with exactly $k$ edges.
We call such a tree a \emph{witness tree} for $e$.

We illustrate the new concepts in Figure~\ref{fig:small-examples} using the wheel graph $G$ with a $7$-cycle $C$ and one central vertex $v$ connected to all vertices in $C$.
Note that any induced forest containing $v$ contains no edge of $C$, that the cycle $C$ can not be covered with only one forest, and that the edges incident to $v$ can not be covered with only two induced forests.
It follows that~$f_k(G) \geq 5$ for $k=1,2,3$ and $f_k(G) \geq 2$ for $k=4,5$, where both inequalities are tight as certified by the covers $\{F_1,\ldots,F_5\}$, respectively $\{F_1,F_2\}$, of the $k$-valid edges with $k$-strong forests given in Figure~\ref{fig:small-examples}.

We discuss other variants of the $k$-strong induced arboricity (where for example \emph{all} edges of the graph have to be covered) in our concluding remarks in Section~\ref{sec:Conclusions}.

\paragraph{Our results:}
The main result of this paper shows that for every graph class $\cC$ of bounded expansion (a notion of sparsity recently introduced by Ne{\v{s}}et{\v{r}}il and Ossona de Mendez~\cite{NO12}) and every natural number $k$, the parameter $f_k$ for graphs in $\cC$ is bounded from above by a constant independent of the order of the graph.
We provide the formal definitions of the following concepts in Section~\ref{sec:preliminaries}, where we also discuss their immediate relations to each other: graph classes of bounded expansion, nowhere dense classes of graphs, as well as tree-width, tree-depth, and adjacent closed vertex-distinguishing number of a graph.


\begin{theorem}\label{thm:main}
 Let $\cC$ be a class of graphs that is of bounded expansion.
 Then for each positive integer $k$ there is a constant $c_k = c(k,\cC)$ such that for each $G \in \cC$ we have $f_k(G) \leq c_k$.
\end{theorem}

\begin{remark}
 Graph classes of bounded expansion include for example all proper minor-closed families and thus also the class of all planar graphs.
 Recall that a graph class $\cC$ is called \emph{minor-closed} if for each $G\in\cC$ any graph $H$ obtained from $G$ by deleting edges or vertices, or by contracting edges is contained in $\cC$.
\end{remark}



As mentioned above, Theorem~\ref{thm:main} answers an open question about vertex-distinguishing numbers of planar graphs~\cite{AHP16}, c.f.\ Corollary~\ref{cor:bounded-expansion-distinguishing}, which follows immediately from Theorem~\ref{thm:main} and the following proposition.

\begin{proposition}[Axenovich~\textit{et al.}~\cite{AHP16}]\label{prop:2-strong-to-distinguishing}
 Let $G$ be a graph whose edges are covered by $m$ $2$-strong forests.
 Let $p_1,\ldots,p_m$ be pairwise relatively prime integers, each at least~$4$.
 Then $\dis[G] \leq p_1p_2 \cdots p_m$.
\end{proposition}


\begin{corollary}\label{cor:bounded-expansion-distinguishing}
 For every graph class $\cC$ of bounded expansion there is an absolute constant $c$ such that for any $G \in \cC$, the adjacent closed vertex-distinguishing number $\dis[G] \leq c$.
 In particular there exists such a constant $c$ for the class of all planar graphs.
\end{corollary}

For $k \in \bN$, let us call a class $\cC$ of graphs \emph{$f_k$-bounded}, if there is a constant $c_k=c(k,\cC)$ such that $f_k(G)\leq c_k$ for each $G\in \cC$, and let us say that $\cC$ is \emph{$f$-bounded} if $\cC$ is $f_k$-bounded for each $k \in \bN$.
Then Theorem~\ref{thm:main} states that each class of bounded expansion is $f$-bounded.
This result however can not be generalized to nowhere dense classes of graphs (another notion of sparsity that includes for example all classes of bounded expansion).
We show in Theorem~\ref{thm:monotonicity}\ref{enumi:nowhere-dense-not-f-bounded} that there exist nowhere dense classes of graphs that are not $f$-bounded, in fact not $f_k$-bounded for any $k \in \bN$.
Indeed, nowhere dense classes of graphs and $f$-bounded classes of graphs are two different extensions of classes of graphs of bounded expansion.
We show in Theorem~\ref{thm:monotonicity}\ref{enumi:f_k-notNowhereDense} that there exist $f$-bounded classes of graphs that are not nowhere dense (and hence not of bounded expansion).
We show in Theorem~\ref{thm:monotonicity}\ref{enumi:monotonicity} that the $k$-strong induced arboricity is in general a highly non-monotone parameter.
We also show in Theorem~\ref{thm:monotonicity}\ref{enumi:1-strong-from-chi},\ref{enumi:monotonicity} some relations between the parameters $k$-strong induced arboricity $f_k(G)$, tree-width $\tw(G)$, tree-depth $\td(G)$, edge arboricity $a'(G)$, and acyclic chromatic number $\chi_{\rm acyc}(G)$.
Recall that the acyclic chromatic number of a graph $G$ is the smallest number of colors in a proper coloring of $G$ in which any two color classes induce a forest. 

\begin{theorem}\label{thm:monotonicity}
 \begin{enumerate}[label = (\roman*)]
 \item For each graph $G$ we have $\log_3(\chi_{\rm acyc}(G)) \leq f_1(G) \leq \binom{\chi_{\rm acyc}(G)}{2}$.\label{enumi:1-strong-from-chi}
 
 \item For any integers $k,n\geq 1$, and for each item below there is a graph $G$ satisfying the listed conditions:\label{enumi:monotonicity}
 
\begin{enumerate}[ref=(\roman{enumi}.\alph*)]
\item $a'(G) = 2$ and $f_k(G) \geq n$,\label{enumii:a-vs-f_1}
\item $f_k(G) \leq 3$ and $f_{k+1}(G) \geq n$,\label{enumii:f_k-smaller-f_k+1}
\item $f_k(G) \geq n$ and $f_{k+1}(G) = 0$,\label{enumii:f_k-larger-f_k+1}
\item $G$ has an induced subgraph $H$ such that $f_k(G) =3$ and $f_k(H) \geq k$,\label{enumii:f_k-subgraphs}
\item $\tw(G) =2$ and $f_k(G) \geq k$,\label{enumii:tw-vs-f_k}
\item $\td(G) =3$ and $f_k(G) \geq k-1$.\label{enumii:td-vs-f_k}
\end{enumerate}

\item There is a class $\cC$ of graphs that is nowhere dense, but not $f_k$-bounded for any $k \geq 1$.\label{enumi:nowhere-dense-not-f-bounded}

\item There is a class $\cC$ of graphs that is $f_k$-bounded for every $k \geq 1$, but not nowhere dense.\label{enumi:f_k-notNowhereDense}

\end{enumerate}
\end{theorem}

We remark that Theorem~\ref{thm:monotonicity}\ref{enumi:1-strong-from-chi} and Theorem~\ref{thm:monotonicity}\ref{enumii:f_k-smaller-f_k+1} together imply that in general $f_k(G)$ can not be bounded from above by a function of $\chi_{\rm acyc}(G)$, that is, Theorem~\ref{thm:monotonicity}\ref{enumi:1-strong-from-chi} can not be extended to $f_k(G)$ for any $k \geq 2$.

\medskip

Theorem~\ref{thm:main} provides the existence of constants bounding $f_k(G)$ for graphs $G$ from any class of bounded expansion.
Next, we give more specific bounds on these constants for special classes. 
Clearly, if $\tw(G) \leq 1$, then $f_k(G) \leq 1$ for every $k$.
However, already for graphs $G$ of tree-width~$2$ finding the largest possible value of $f_k(G)$ for $k \geq 2$ is non-trivial.
We show in particular that $f_1(G)\leq\binom{\tw(G)+1}{2}$ for any graph $G$, which is best-possible, since~$\tw(K_{t+1})=t$ and $f_1(K_{t+1})=\binom{t+1}{2}$, and that $f_2(G) \leq 3\binom{\tw(G)+1}{3}$ for any graph $G$ with $\tw(G) \geq 2$, which is best-possible when $\tw(G)=2$, as certified by $G$ being $K_3$ with a pendant edge at each vertex.

\begin{theorem}\label{thm:tw3}
 For every graph $G$ of tree-width $t\geq 2$, we have that $f_1(G) \leq \binom{t+1}{2}$ and $f_2(G)\leq 3\binom{t+1}{3}$.
\end{theorem}

The next theorem shows that the parameter $f_k$ is bounded for graphs of tree-depth $d$ by a polynomial in $d$ as well as a polynomial in $k$.

\begin{theorem}\label{thm:tree-depth}
 For all positive integers $k$, $d$ and any graph $G$ of tree-depth~$d$, $f_k(G) \leq (2k)^{d}$.
 If $d\geq k+1$, then $f_k(G) \leq (2k)^{k+1}\binom{d}{k+1}$.
 Moreover $f_1(G)\leq \binom{d}{2}$.
\end{theorem}

Ne{\v{s}}et{\v{r}}il and Ossona de Mendez~\cite{NO03} prove that for each minor-closed class $\cC$ of graphs that is not the class of all graphs there is a constant $x$ such that each graph in $\cC$ has acyclic chromatic number at most $x$.
We show how to bound $f_2$ in terms of $x$.

\begin{theorem}\label{thm:acyclic}
 For every minor-closed class of graphs $\cC$ whose members have acyclic chromatic number at most $x$, we have that for every $G\in\cC$,
 \[
  f_2(G) \leq \begin{cases}
    \binom{x}{2}( 3\binom{x}{2}+1),&\text{ if } x\leq 9,\\
    \binom{x}{2}(12x + 1),&\text{ if } x\geq 9.
    \end{cases}
 \]
 For every planar graph $G$ we have $f_2(G) \leq 310$.
\end{theorem}


\paragraph{Organization of the paper:}
In Section~\ref{sec:preliminaries} we review the concepts tree-width and tree-depth of a graph, graph classes of bounded expansion, and nowhere dense classes of graphs by giving the formal definitions and discussing the interrelations.
We also define the adjacent closed vertex-distinguishing number of a graph in Section~\ref{sec:preliminaries}.
We prove Theorem~\ref{thm:monotonicity} in Section~\ref{sec:monotonicity}.
We consider graphs of bounded tree-width in Section~\ref{sec:tree-width} and prove Theorem~\ref{thm:tw3} in that section.
The proofs of the bounds on $f_k$ in terms of the acyclic chromatic number and the proof of Theorem~\ref{thm:acyclic} are given in Section~\ref{sec:minor-closed}.
Graphs of bounded tree-depth, and more general classes of graphs, are considered in Section~\ref{sec:tree-depth}, where we prove Theorem~\ref{thm:tree-depth}.
We prove the main theorem, Theorem~\ref{thm:main}, in Section~\ref{sec:main}.
Finally we summarize our results, state some open questions, and discuss other variants of the strong induced arboricity in Section~\ref{sec:Conclusions}.


\section{Preliminaries}\label{sec:preliminaries}

\paragraph{Notation:}
For an integer $s \geq 1$ we denote by $[s]$ the set $\{1,\ldots,s\}$ of the first $s$ natural numbers.
For all further standard graph theoretic notions and notations we refer the reader to the book of West~\cite{West}.

\paragraph{Graph parameters and graph class properties:}
In this section we formally define the parameters and properties for graphs and graph classes mentioned in the introduction, namely tree-width, tree-depth and adjacent closed vertex-distinguishing number of graphs, and graph classes of bounded expansion and nowhere dense classes of graphs.
Moreover, we summarize the relationships between those parameters and classes.
We shall follow the notions used by Ne{\v{s}}et{\v{r}}il and Ossona de Mendez~\cite{NO06,NO08}, see also~\cite{GPP13,NO12}.

\begin{description}
 \item{Tree decompositions:~~}
  A \emph{tree decomposition} of a graph $G$ with vertex set $V(G) = \{v_1,\ldots,v_n\}$ is a pair $(T,B)$ of a tree $T$ and a set $B = \{B_1,\ldots,B_n\}$ of non-empty subsets of $V(T)$, the vertex set of $T$, such that \textbf{(A)} for $i =1,\ldots,n$ the vertices in $B_i$ induce a connected subgraph of $T$ and \textbf{(B)} for each edge $v_iv_j$ in $G$ we have $B_i \cap B_j \neq \emptyset$.
  Intuitively speaking, a tree decomposition of $G$ is a representation of a supergraph of $G$ as the intersection graph of some subtrees of the tree $T$.
  The \emph{width} of a tree decomposition $(T,B)$ is defined as $\max_{v \in V(T)}|\{i \in [n] \mid v \in B_i\}| - 1$.
  
 \item{Tree-width:~~}
  The \emph{tree-width} of a graph $G$, denoted by $\tw(G)$, is the smallest $t$ such that $G$ has a tree decomposition of width $t$.
  Equivalently, a graph $G$ has tree-width $t$ if $t$ is the smallest number for which $G$ is a spanning subgraph of a $t$-tree $H$, where a \emph{$t$-tree} can be recursively defined as follows:
  The graph $K_{t+1}$ is a $t$-tree, and if $H$ is a $t$-tree, $C$ is a clique of order $t$ in $H$ and $H'$ arises from $H$ by adding a new vertex whose neighborhood is $C$, then $H'$ is also a $t$-tree~\cite{Bod98}.
  Note that the graphs of tree-width $1$ are the forests with at least one edge.
  

 \item{Tree-depth:~~}
  The \emph{transitive closure} of a rooted tree $T$ with a root $r$ is the graph obtained from $T$ by adding every edge $uv$ such that $v$ is on the $u$-$r$-path of $T$.
  A rooted tree has \emph{depth $d$} if the largest number of vertices on a path to the root is $d$. 
  Now, a graph $G$ has \emph{tree-depth} $d$, denoted by $\td(G) = d$, if $d$ is the smallest integer such that each connected component of $G$ is a subgraph of the transitive closure of a rooted tree of depth $d$.

 \item{Tree-depth coloring:~~}
  A \emph{$p$-tree-depth coloring} of a graph $G$ is a vertex coloring such that each set of $p'$ color classes, $p' \leq p$, induces a subgraph of $G$ with tree-depth at most~$p'$.
  So a $1$-tree-depth coloring is exactly a proper coloring of $G$, while a $2$-tree-depth coloring is a proper coloring of $G$ in which any two color classes induce a star forest (a graph of tree-depth at most~$2$).
  Let $\chi_p(G)$ be the minimum number of colors needed in a $p$-tree-depth coloring of $G$.
  Then $\chi(G) = \chi_1(G)$ and 
  $\chi_p(G) \leq \td(G)$ for any $p\geq 1$~\cite{NO08}.

 \item{Shallow minor:~~}
  For a graph $H$ and a non-negative integer $d$, we say that a graph $G$ with vertex set $\{v_1,\ldots,v_n\}$ is a \emph{$d$-shallow minor} of $H$ if there exist pairwise disjoint subsets of vertices $B_1,\ldots,B_n$ in $H$ such that \textbf{(A)} for $i=1,\ldots,n$ the vertices in $B_i$ induce a connected subgraph of $H$ with radius at most $d$ and \textbf{(B)} for each edge $v_iv_j$ of $G$ there is an edge in $H$ with one endpoint in $B_i$ and one endpoint in $B_j$.
  Here a graph has radius at most $d$ if there is a vertex that is within distance $d$ to all other vertices.
  The class of all $d$-shallow minors of graphs from a class $\cC$ of graphs is denoted by $\cC \nabla d$.
  Note that $\cC \nabla 0$ is exactly the class of all subgraphs of graphs in $\cC$.
  
 \item{Bounded expansion:~~}
  A class $\cC$ of graphs is of \emph{bounded expansion} if for every non-negative integer $d$ there is a constant $a_d = a(d,\cC)$, such that every $d$-shallow minor of a graph in $\cC$ has at most $a_d$ times as many edges as vertices.
  That is, for every $G \in \cC \nabla d$ we have $|E(G)| \leq a_d |V(G)|$.
  Equivalently, $\cC$ is of bounded expansion if for each positive integer $p$ there is a constant $b_p = b(p,\cC)$ such that for each $G\in\cC$ we have $\chi_p(G) \leq b_p$~\cite{NO12}.
  
 \item{Nowhere dense:~~}
  A class $\cC$ of graphs is \emph{nowhere dense} if for each non-negative integer $d$ we have that $\cC \nabla d$ is not the class of all graphs.
  That is, for each $d$ there exists a graph $G_d = G(d,\cC)$ such that $G_d$ is not a $d$-shallow minor of any graph in $\cC$, i.e., $G_d \notin \cC \nabla d$.
  
 \item{Adjacent closed vertex-distinguishing number:~~}
  For a graph $G$, an assignment of positive integers to its vertices is called \emph{distinguishing} if the sum of the labels in the closed neighborhood $N[v]$ of any vertex $v$ differs from the sum in the closed neighborhood of any of the neighboring vertices $u$ of $v$, unless $N[u]=N[v]$.
  Note that $N[u] = N[v]$ if and only if $uv$ is an edge that is not $2$-valid.
  The smallest positive integer $\ell$ such that there is a distinguishing labeling of $G$ with labels in $\{1,\ldots,\ell\}$ is called \emph{adjacent closed vertex-distinguishing number} of $G$, denoted $\dis[G]$. 
\end{description}

Tree-width is an important graph parameter and the corner stone of the Graph Minor Project~\cite{RS84}.
It is well-known that tree-width is monotone under taking graph minors, namely, if $H$ is a minor of $G$, then $\tw(H) \leq \tw(G)$, see for example~\cite{Bod98}.
Thus, by the Graph Minor Theorem~\cite{RS04}, the class of all graphs of tree-width at most $d$ is characterized by finitely many excluded minors.
On the other hand, for example planar graphs are characterized by two excluded minors, but there are planar graphs of arbitrarily large tree-width.
The equivalent definition of tree-width~$t$ graphs as subgraphs of $t$-trees and the recursive definition of $t$-trees is very convenient for inductive proofs.
For example, following the construction sequence of a $t$-tree $H$, we easily see that $\chi(H) = t+1$ and thus $\chi(G) \leq t+1$ whenever $\tw(G) \leq t$.

The tree-depth is somehow a more restrictive variant of the tree-width, measuring how far a graph is from being a star rather than a tree.
For any fixed graph $G$ we have that $\tw(G) \leq \td(G)-1$.
On the other hand, if $G$ has tree-depth~$d$, then the longest path in $G$ has at most $2^d - 1$ vertices.
In particular, even graphs of tree-width~$1$ can have arbitrarily large tree-depth~\cite{NO08}.
Just like tree-width, tree-depth is a minor-monotone graph parameter~\cite{NO12}.

Bounded expansion was introduced by Ne\v{s}et\v{r}il and Ossona de Mendez~\cite{NO08} as a notion of sparsity of a graph class $\cC$.
Ne{\v{s}}et{\v{r}}il and Ossona de Mendez~\cite{NO06,NO08,NOW12} proved that several classes of graphs are of bounded expansion, such as proper minor-closed classes, classes of graphs with an excluded topological minor, or classes where the graphs have bounded degree, bounded book thickness, or bounded queue number.


Using shallow minors, Ne\v{s}et\v{r}il and Ossona de Mendez~\cite{NO12} split all graph classes into somewhere dense and nowhere dense classes.
Even though they are generalizing classes of bounded expansion, nowhere dense classes still contain sparse graphs in the sense that for every $\varepsilon > 0$ the $n$-vertex graphs in a nowhere dense class $\cC$ have $O(n^{1+\varepsilon})$ edges.
Nowhere dense graph classes have nice algorithmic properties and admit several (seemingly unrelated) characterizations~\cite{GKS13}.

While for an analogous notion $\dis(G)$ with open neighborhoods considered instead of closed neighborhoods, it is known that there is a constant $c$ such that $\dis(G)\leq c$ for any planar graph $G$, as noted by Norine, see~\cite{GB}, it was not known whether the adjacent closed vertex-distinguishing number $\dis[G]$ is bounded by a universal constant for all planar graphs.
Corollary~\ref{cor:bounded-expansion-distinguishing} answers this question in the affirmative, since the class of all planar graphs is a graph class of bounded expansion.

\medskip

In Figure~\ref{fig:graph-class-properties} we display several graph class properties and depict their relations as they follow from the discussion above, Theorem~\ref{thm:main}, and Theorem~\ref{thm:monotonicity}\ref{enumi:nowhere-dense-not-f-bounded}.
All relations in Figure~\ref{fig:graph-class-properties} are strict, that is, whenever there is an arrow from $A$ to $B$ then every graph class with property $A$ has also property $B$, but there are graph classes that have $B$ but not $A$.
Recall that a graph class $\cC$ is $f$-bounded if for every $k \geq 1$ there exists a constant $c_k = c(k,\cC)$ such that $f_k(G) \leq c_k$ for every $G \in \cC$.

\begin{figure}[tb]
 \centering
 \includegraphics{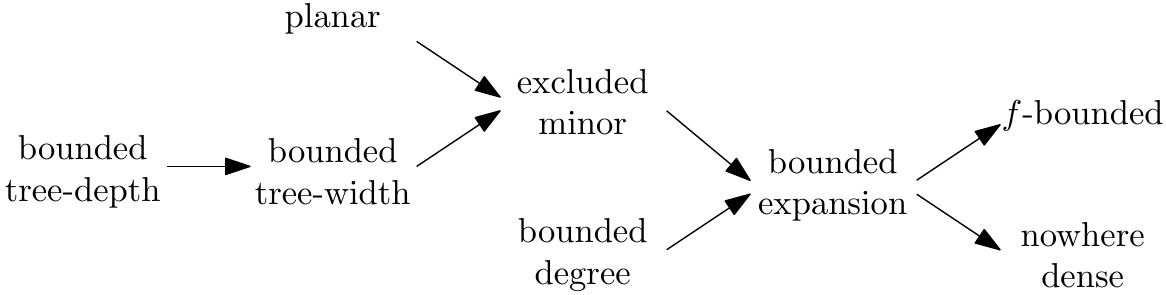}
 \caption{Overview of properties of some graph classes considered in this paper. Arrows go from more restrictive properties to more general properties.}
 \label{fig:graph-class-properties}
\end{figure}

\section{Proof of Theorem~\ref{thm:monotonicity}}
\label{sec:monotonicity}

In this section we prove the general properties of the parameter $f_k$ listed in Theorem~\ref{thm:monotonicity}.

\begin{proof}[Proof of Theorem~\ref{thm:monotonicity}\ref{enumi:1-strong-from-chi}]
For the first inequality, consider a covering of $E(G)$ with $x=f_1(G)$ induced forests $F_1, \ldots, F_x$ and for each forest a proper $2$-coloring of its vertices.
Let $c_1, \ldots, c_x$ be colorings of $V(G)$ in colors $\{0,1,2\}$ such that $c_i(v) =0$ if $v\not \in V(F_i)$, $c_i(v) = 1$ if $v$ is from the first color class of $F_i$, and $c_i(v) = 2$ if $v$ is from the second color class of $F_i$. 
Let a coloring $\varphi$ of $V(G)$ be defined as $\varphi(v) = (c_1(v), \ldots, c_x(v))$, $v\in V(G)$. To see that $\varphi$ is an acyclic coloring assume that two color classes $(a_1, \ldots, a_x)$, $(b_1, \ldots, b_x)$ induce a cycle $C$.
Let $e$ be an edge of $C$.
It is in some $F_i$ and hence $\{a_i, b_i\} = \{1,2\}$.
Thus the $i^{\rm th}$ coordinate of $\varphi$ in the cycle $C$ alternates between $1$ and $2$. This implies that all edges of $C$ belong to $F_i$, a contradiction since $F_i$ is acyclic. 
For similar reasons $\varphi$ is proper.
Thus $\varphi$ is an acyclic coloring, proving that $\chi_{\rm acyc}(G) \leq 3^x$ and thus $\log_3(\chi_{\rm acyc}(G)) \leq x = f_1(G)$.

For the second inequality, consider an acyclic proper coloring of $G$ using $\chi_{\rm acyc}(G)$ colors.
 For every pair of colors $c_1$, $c_2$ the subgraph of $G$ induced by the vertices of color $c_1$ or $c_2$ is an induced forest in $G$.
 Moreover, every edge of $G$ is contained in exactly one such induced forest.
 Hence, by removing all isolated vertices from each such forest, we get $f_1(G) \leq \binom{\chi_{\rm acyc}}{2}$.
\end{proof}

 \begin{proof}[Proof of Theorem~\ref{thm:monotonicity}\ref{enumii:a-vs-f_1}]
 Let $R^{-1}(t)$ denote the smallest number of colors needed to color $E(K_t)$ without monochromatic triangles.
 By Ramsey's Theorem~\cite{Ram30,W97} we have $R^{-1}(t) \to \infty$ as $t \to \infty$.
 Choose $t$ sufficiently large such that $R^{-1}(t) \geq n^2$ and, additionally, $t \geq \max\{k,3\}$.
 
 Let $G$ be the graph obtained from $K_t$ by subdividing each edge once.
 For an edge $e$ in $K_t$ let $e_1$ and $e_2$ denote the two corresponding edges in $G$.
 Split $G$ into two subgraphs $G_1$ and $G_2$ where $G_i$ contains all edges $e_i$, $e\in E(K_t)$, $i=1,2$.
 Then $E(G) = E(G_1)\cup E(G_2)$ and, for $i=1,2$, each component of $G_i$ is a star with center at an original vertex of $K_t$.
 Therefore $a(G) \leq 2$ and as $t \geq 3$, we have $a(G) = 2$.
 We remark that each component of $G_i$ ($i=1,2$) is induced, but as long as $t \geq 4$ one of $G_1,G_2$ induces a $6$-cycle, see Figure~\ref{fig:monoK3} (left part).

 Let $N = f_1(G)$ and consider induced forests $F_1,\ldots,F_N$ covering all edges of $G$. 
 We consider the following edge-coloring of $K_t$.
 If there is an $i$, $1\leq i\leq N$, with $e_1$, $e_2\in E(F_i)$, then color the edge $e$ with color $i$ (choose an arbitrary such $i$).
 Otherwise there are $i$ and $j$, $1\leq i<j\leq N$, with $e_1$, $e_2\in E(F_i)\cup E(F_j)$, $i\neq j$, and we color the edge $e$ with color $\{i,j\}$ (choose an arbitrary such pair).
 This coloring uses at most $N + \binom{N}{2}=\binom{N+1}{2}$ colors.
 We claim that there are no monochromatic triangles under this coloring.
 Indeed there is no triangle in color $i$, $1\leq i\leq N$, since $F_i$ contains no cycle, and there is no triangle in color $\{i,j\}$, $1\leq i<j\leq N$, since $F_i$ and $F_j$ are induced. Therefore $\binom{f_1(G)+1}{2} = \binom{N+1}{2} \geq R^{-1}(t) \geq n^2$.
 This shows that $f_1(G) \geq n$, since $\binom{n}{2}<n^2 \leq R^{-1}(t)$.
 Moreover, as $t \geq k$ every edge in $G$ is $k$-valid and thus $f_k(G) \geq f_1(G) \geq n$.
\end{proof}

\begin{proof}[Proof of Theorem~\ref{thm:monotonicity}\ref{enumii:f_k-smaller-f_k+1}]
 Like in the proof of part~\ref{enumii:a-vs-f_1}, let $R^{-1}(t)$ denote the smallest number of colors needed to color $E(K_t)$ without monochromatic triangles.
 By Ramsey's Theorem~\cite{Ram30,W97} we have $R^{-1}(t) \to \infty$ as $t \to \infty$.
 Choose $t$ sufficiently large such that $R^{-1}(t) \geq n^2$ and, additionally, $t\geq 2k+2$.
 
 Let $G$ be obtained from $K_t$ by subdividing each edge twice and choosing for each original edge of $K_t$ one of its subdivision vertices and adding $k-1$ pendant edges to this vertex, see Figure~\ref{fig:monoK3} (middle part) when $k=2$.
 Observe that all edges of $G$ are $k$-valid and $(k+1)$-valid.
 
 First we shall show that $f_k(G)\leq 3$ by finding $3$ $k$-strong forests covering all edges of $G$.
 For an edge $e$ in $K_t$ let $e_1$, $e_2$, $e_3$ denote the subdividing edges in $G$, with $e_2$ the middle one.
 Let $T_1$ be the subgraph consisting of all edges $e_2$, $e\in E(K_t)$, and all edges adjacent to $e_2$ different from $e_1$ and $e_3$ (the pendant edges).
 Then $T_1$ is an induced forest and each component of $T_1$ is a star on $k$ edges.
 Since $t\geq 2k+2$, we can choose an orientation of $K_t$ such that each vertex has out-degree and in-degree at least $k$.
 Indeed, if $t$ is odd we find such an orientation by following an Eulerian walk, if $t$ is even, we find such an orientation of $K_{t-1}$ as before and orient the edges incident to the remaining vertex $x$ such that at least $k$ of these edges are in-edges at $x$ and at least $k$ of them are out-edges at $x$. 
 For each edge $e=uv$ in $K_t$ that is oriented from $u$ to $v$ put the edge in $\{e_1,e_3\}$ that is incident to $u$ into $T_2$ and the other edge from $\{e_1,e_3\}$ into $T_3$.
 Then $T_2$ and $T_3$ are induced forests and each component of $T_2$ and $T_3$ is a star on at least $k$ edges.
 Moreover each edge of $G$ is contained in $E(T_1)\cup E(T_2)\cup E(T_3)$.
 Therefore $f_k(G)\leq 3$.
 
 \medskip
 Next, we prove that $f_{k+1}(G)\geq n$. 
 Let $N = f_{k+1}(G)$ and consider $(k+1)$-strong forests $F_1,\ldots,F_N$ covering all edges of $G$. 
 For each edge $e$ of $K_t$, if $F_i$ contains $e_2$, then it contains either $e_1$ or $e_3$ as well, since each component of $F_i$ has at least $k+1 \geq 2$ edges.
 We consider the following edge-coloring of $K_t$.
 If there is an $i$, $1\leq i\leq N$, such that $e_1$, $e_2$, $e_3\in E(F_i)$, then color the edge $e$ with $i$ (choose an arbitrary such $i$).
 Otherwise there are distinct $i$, $j$, $1\leq i,j\leq N$, such that, without loss of generality, $e_1$, $e_2\in E(F_i)$ and $e_3\in E(F_j)$.
 In this case color the edge $e$ with the pair $(i,j)$ (choose an arbitrary such pair).
 This coloring uses at most $N + N(N-1)=N^2$ colors.
 We claim that there are no monochromatic triangles under this coloring.
 Indeed, for any $i$ and $j$ there is no triangle in color $i$, $1\leq i\leq N$, since $F_i$ contains no cycle, and there is no triangle in color $(i,j)$, $1\leq i,j\leq N$, since $F_i$ and $F_j$ are induced.
 See Figure~\ref{fig:monoK3} (right part) in case $k=4$.
 \begin{figure}
 \centering
 \includegraphics{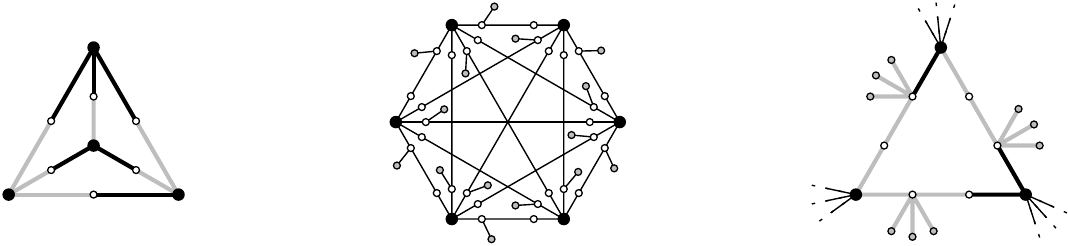}
 \caption{Illustrations of the proofs of Theorem~\ref{thm:monotonicity}\ref{enumii:a-vs-f_1} and~\ref{enumii:f_k-smaller-f_k+1}.
  Left: The graph $G$ obtained from $K_4$ and a partition into two forests $G_1,G_2$.
  Note that both forests induce $6$-cycles.
  Middle: The graph $G$ obtained from $K_6$ for $k=2$.
  Right: A subgraph corresponding to a monochromatic triangle in $K_t$. 
  Note that the gray forest is not induced.}
 \label{fig:monoK3}
 \end{figure}
 Therefore the number of colors is at least $R^{-1}(t)$ and at most $N^2 =f_{k+1}^2(G)$. Thus $f_{k+1}(G) \geq \sqrt{R^{-1}(t)} \geq n$.
\end{proof}

\begin{proof}[Proof of Theorem~\ref{thm:monotonicity}\ref{enumii:f_k-larger-f_k+1}]
 Consider the graph $G$ formed by taking the union of a clique on $n+1$ vertices and a path of length $k-1$ that shares an endpoint with the clique. 
 Then we see that all edges of $G$ incident to the path are $k$-valid. However, no two edges of the clique could be in the same induced forest, thus $f_k(G)\geq n$. 
 On the other hand, since each induced tree in $G$ contains at most one edge from the clique, it could have at most $k$ edges. Thus there are no $(k+1)$-valid edges and $f_{k+1}(G) = 0$.
\end{proof}

\begin{proof}[Proof of Theorem~\ref{thm:monotonicity}\ref{enumii:f_k-subgraphs} and~\ref{enumii:tw-vs-f_k}]
 Consider the graph $G$ shown in Figure~\ref{fig:zigzag}.
 We see from Figure~\ref{fig:zigzag} that $G$ is covered by three large induced trees (a bold, a solid, and a dashed path) and thus $f_k(G)\leq 3$.
 Let $H$ be its induced subgraph formed by the bold vertices shown in Figure~\ref{fig:zigzag}. 
 We see that $H$ is formed by a path $u_1, u_2, \ldots , u_{2k}$ and independent vertices $w_1, w_2, \ldots, w_{2k-1}$ such that $w_i$ is adjacent to $u_i$ and $u_{i+1}$.
 Then consider the matching in $H$ formed by the edges $u_iw_i$, $k\leq i\leq 2k-1$, and an induced tree $T_i$
 on vertex set $\{u_{i-k+1},\ldots,u_i,w_i\}$ in $H$ of size $k$ containing $u_iw_i$, $k\leq i \leq 2k-1$. 
 We see that the trees $T_{k},\ldots,T_{2k-1}$ are distinct and their pairwise union induces a triangle in $H$.
 Thus no two of them can belong to the same $k$-strong forest in $H$. Hence $f_k(H)\geq k$.
 This proves Theorem~\ref{thm:monotonicity}\ref{enumii:f_k-subgraphs}.
 In addition, $\tw(H)=2$.
 This proves Theorem~\ref{thm:monotonicity}\ref{enumii:tw-vs-f_k} (where $H$ plays the role of $G$ from the theorem).
 \begin{figure}
 \centering
 \includegraphics{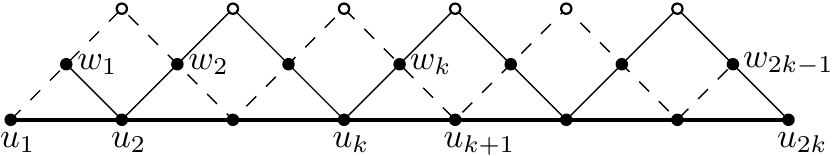}
 \caption{A graph $G$ illustrating the proof of Theorem~\ref{thm:monotonicity}\ref{enumii:f_k-subgraphs} and~\ref{enumii:tw-vs-f_k} in case $k=4$ with the subgraph $H$ induced by bold vertices. The marked paths (one bold, one solid, one dashed) form three induced forests covering all the edges of $G$.}
 \label{fig:zigzag}
 \end{figure}
\end{proof}

\begin{proof}[Proof of Theorem~\ref{thm:monotonicity}\ref{enumii:td-vs-f_k}]
 \begin{figure}
 \centering
 \includegraphics{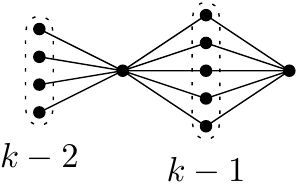}
 \caption{A graph $G$ illustrating the proof of Theorem~\ref{thm:monotonicity}(ii.f).}
 \label{fig:td2}
\end{figure}
 Consider the graph $G$ shown in Figure~\ref{fig:td2}. 
 Then $\td(G)=3$ (look at the cut vertex as a root of the underlying tree) and $f_k(G) = k-1$, since each edge incident to the rightmost vertex is $k$-valid but no $k$-strong forest contains two of these edges.
\end{proof}

\begin{proof}[Proof of Theorem~\ref{thm:monotonicity}\ref{enumi:nowhere-dense-not-f-bounded}]
 Let $\cC$ be any infinite class of connected graphs $G_n$, $n \geq 1$, where $G_n$ has girth and chromatic number at least $n$.
 Due to the girth condition $\cC$ is nowhere dense.
 In fact, for $n>6d+3$ any $d$-shallow minor of $G_n$ contains no triangle and hence for any $d$ any large enough graph with a triangle is not contained in $\cC \nabla d$.
 On the other hand, for $n \geq k+2$ each edge of $G_n$ is $k$-valid.
 As $\chi_{\rm acyc}(G_n) \geq \chi(G_n) \geq n$, we conclude from part~\ref{enumi:1-strong-from-chi} that the graphs in $\cC$ have unbounded $k$-strong induced arboricity.
\end{proof}

\begin{proof}[Proof of Theorem~\ref{thm:monotonicity}\ref{enumi:f_k-notNowhereDense}]
 Consider the complete bipartite graph $K_{n,n}$, $n \geq 1$, and let $G_n$ be the graph obtained from $K_{n,n}$ by subdividing each edge once.
 We claim that for any integers $k,n \geq 1$ we have $f_k(G_n) \leq 2$, and moreover that the graph class $\cC = \{G_n \mid n \in \bN\}$ is not nowhere dense.
 For the latter we shall find a constant $d$ such that the class $\cC \nabla d$ of all $d$-shallow minors of graphs in $\cC$ contains all complete graphs and therefore all graphs.
 In fact for any fixed $n$ we shall show that $K_n$ is a $2$-shallow minor of $G_n \in \cC$.
 Let $M = \{u_iv_i \mid i \in [n]\}$ be a perfect matching in $K_{n,n}$ where all $u_i$ are in the same partite set.
 For $i = 1,\ldots,n$ set $B_i = \{u_i,v_i\} \cup N(u_i)$, where $N(u_i)$ denotes the set of all neighbors of $u_i$ in $G_n$.
 Then $B_1,\ldots,B_n$ partition the vertex set of $G_n$ into disjoint subsets, each inducing a connected subgraph of $G_n$ with radius~$2$.
 Moreover, for any $i \neq j$ there is an edge in $G_n$ between $G_n[B_i]$ and $G_n[B_j]$.
 Thus $K_n$ is a $2$-shallow minor of $G_n$, proving that $\cC$ is not nowhere dense.

 Next we prove that, for any integers $k, n \geq 1$ we have $f_k(G_n) \leq 2$.
 To this end, we construct two maximum induced trees in $G_n$ covering all edges of $G_n$.
 Clearly, $|V(G_n)| = 2n+n^2$ and we claim that a largest induced tree in $G_n$ contains exactly $n+1 + n^2 = |V(G_n)| - (n-1)$ vertices.
 Let $\mu(G_n)$ denote the smallest number of vertices in $G_n$ whose deletion makes the graph acyclic.
 (That is, $\mu(G_n)$ denotes the size of a minimum feedback vertex set -- see~\cite{FPR99} for a survey on feedback set problems.)
 We shall prove by induction on $n$ that $\mu(G_n) \geq n-1$.
 In fact, for $n = 1$, $G_n$ is a tree itself and thus $\mu(G_1) = 0$.
 For $n \geq 2$, consider an $8$-cycle in $G_n$ consisting of four original vertices $v_1,v_2,v_3,v_4$ of $K_{n,n}$, $v_1$, $v_2$ from one bipartition class and $v_3$, $v_4$ from the other, and the four subdivision vertices corresponding to the four edges $v_1v_3$, $v_1v_4$, $v_2v_3$ and $v_2v_4$ in $K_{n,n}$.
 At least one of these eight vertices has to be deleted to make the graph acyclic, say it is one of $v_1$, $v_3$, or the vertex $x$ subdividing edge $v_1v_3$.
 Then $G_n - (N[v_1] \cup N[v_3])$ is isomorphic to $G_{n-1}$ and thus at least $\mu(G_{n-1})$ further vertices have to be deleted.
 Hence by induction we get $\mu(G_n) \geq \mu(G_{n-1}) + 1 \geq (n-2) + 1 = n-1$, as desired.
 Thus any induced tree in $G_n$ has at most $n^2 + 2n - (n-1)$ vertices.
 
 \begin{figure}
  \centering
  \includegraphics{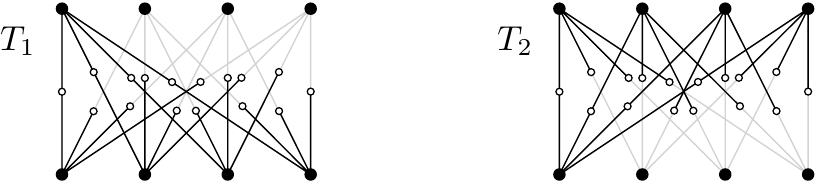}
  \caption{Two maximum induced trees $T_1$, $T_2$ covering all edges of the graph $G_n$ ($n = 4$) obtained from $K_{n,n}$ by subdividing every edge once.}
  \label{fig:K_nn-trees}
 \end{figure}

 On the other hand, one obtains a maximum induced tree $T_1$ by deleting $n-1$ original vertices of $K_{n,n}$ that belong to the same bipartition class, see Figure~\ref{fig:K_nn-trees}.
 Deleting $n-1$ vertices from the other bipartition class gives symmetrically a maximum induced tree $T_2$.
 Finally, observe that $T_1$ and $T_2$ together cover all edges of $G_n$, which certifies that $f_k(G_n) \leq 2$ for $k \leq n+1 + n^2$.
 For $k > n+1 + n^2$ no edge of $G_n$ is $k$-valid and thus $f_k(G_n) = 0$. 
%
\end{proof}

 \section{Graphs of bounded tree-width}\label{sec:tree-width}

We start with a list of properties of graphs of tree-width~$2$.
Then we shall prove that $f_2(G)\leq 3$ for any graph $G$ of tree-width~$2$.
This is the main part of the proof, where we shall use the definition of tree-width~$2$ graphs as being exactly the partial $2$-trees.
Then, we shall do an easy reduction argument using tree decompositions to express the upper bound from Theorem~\ref{thm:tw3} on $f_2(G)$ for graphs $G$ of larger tree-width.


\subsection{Properties of graphs with tree-width 2 and observations}

Consider any fixed graph $G$ of tree-width~$2$.
Firstly, $G$ contains no subdivision of $K_4$~\cite{Bod88}.
(In fact, this property characterizes tree-width~$2$ graphs.)
Moreover, it is well-known (see for example~\cite{Ree03}) that as long as $|V(G)| \geq 3$, there is a $2$-tree $H$ with $G \subseteq H$ and $V(H) = V(G)$.
Let us fix such a $2$-tree $H$.
Every edge of $H$ is in at least one triangle of $H$.
Consider the partition $E(H) = E_\text{in}(H) \dot\cup E_\text{out}(H)$ of the edges of $H$, where $E_\text{out}(H)$ consists of those edges that are contained in only one triangle of $H$, called the \emph{outer edges of $H$}.
Respectively, $E_\text{in}(H)$ consists of those edges that are contained in at least two triangles of $H$,
called the \emph{inner edges of $H$}.
Note, if $H$ is outerplanar, every edge is in at most two triangles, and our definition corresponds to the usual partition into outer and inner edges of an outerplanar embedding of $H$.

The following two statements can be easily proved by induction on $|V(H)|$.
Indeed, both statements hold with ``if and only if'' and are maintained in the construction sequence of the $2$-tree $H$.

\begin{enumerate}[label = \textbf{(P\arabic*)}]
 \item If $v \in V(H)$ is incident to two outer edges in the same triangle of $H$, then $\deg_H(v) = 2$.\label{enum:degree-2-vertices}

 \item If $uw \in E_\text{in}(H)$, then $H - \{u,w\}$ is disconnected.\label{enum:inner-edges-disconnect}
\end{enumerate}

It is easy to see that for any $2$-connected graph $F$ with $|V(F)| \geq 4$ and for any two vertices $u,w \in V(F)$ we have the following:

\begin{enumerate}[label = \textbf{(P\arabic*)}, start = 3]
 \item For every connected component $K$ of $F - \{u,w\}$ we have $N(u) \cap V(K) \neq \emptyset$ and $N(w) \cap V(K) \neq \emptyset$.\label{enum:neighbors-in-components}
 \item The graph $F - \{u,w\}$ is connected if and only if the graph $F'$ obtained from $F$ by identifying $u$ and $w$ into a single vertex is $2$-connected.\label{enum:cutvertex-after-contraction}
\end{enumerate}

Now if $G$ is a $2$-connected graph of tree-width~$2$ and $H$ is a $2$-tree with $G \subseteq H$ and $V(H) = V(G)$, then we have the following properties.
\begin{enumerate}[label = \textbf{(P\arabic*)}, start = 5]
 \item $E_\text{out}(H) \subseteq E(G)$\label{enum:outer-are-in-G}
 \item For every $e \in E_\text{out}(H)$ the graph $G / e$ obtained from $G$ by contracting edge $e$ is $2$-connected.\label{enum:contracting-outer-is-good}
\end{enumerate}

To see~\ref{enum:outer-are-in-G}, consider any edge $e = uw $ in $E_\text{out}(H)$.
As $G$ is $2$-connected, there exists a cycle $C$ in $G$ through $u$ and $w$.
If $e \in E(C)$, then $e \in E(G)$ and we are done.
Otherwise, in $H$, edge $e$ is a chord of cycle $C$, splitting it into two cycles $C_1$ and $C_2$.
As $H$ is a chordal graph, $C_1$ and $C_2$ are triangulated, i.e., $e$ is contained in a triangle with vertices in $C_1$ and another triangle with vertices in $C_2$.
Thus $e \in E_\text{in}(H)$, a contradiction to $e \in E_\text{out}(H)$.

To see~\ref{enum:contracting-outer-is-good}, consider any outer edge $e = uw$ of $H$.
By~\ref{enum:cutvertex-after-contraction} we have that $G / e$ is $2$-connected if $G - \{u,w\}$ is connected.
Assume for the sake of contradiction that $G - \{u,w\}$ is disconnected and let $K_1$, $K_2$ be two connected components of $G - \{u,w\}$.
Then by~\ref{enum:neighbors-in-components} for $i = 1,2$ we have $N(u) \cap V(K_i) \neq \emptyset$ and $N(w) \cap V(K_i) \neq \emptyset$.
Hence we can find a cycle $C$ in $H$ for which $e = uw$ is a chord by going from $u$ to $w$ through $K_1$ and from $w$ to $u$ through $K_2$.
As before, it follows that $e \in E_\text{in}(H)$, a contradiction to $e \in E_\text{out}(H)$.
Hence, $G / e$ is $2$-connected.

\medskip

Finally, let us characterize the edges of $G$ that are not $2$-valid.
An edge $uv$ of $G$ is called a \emph{twin edge} if $N[u] = N[v]$, i.e., if the closed neighborhoods of $u$ and $v$ coincide.
Observe that twin edges are exactly the edges that are not $2$-valid.

\begin{enumerate}[label = \textbf{(P\arabic*)}, start = 7] 
\item If $G$ is $2$-connected, $\tw(G) = 2$, and $xy$ is a twin edge in $G$, then $G$ is a $2$-tree consisting of $r$ triangles, $r \geq 1$, all sharing the common edge $xy$. \label{enum:twins}
\end{enumerate}

To prove~\ref{enum:twins}, let $H$ be a $2$-tree with $G \subseteq H$ and $V(H) = V(G)$.
 Consider the set $S = N(x) - y = N(y) - x$.
 As $G$ is $2$-connected, we have $|S| \geq 1$.
 We claim that for each $w \in S$ the edges $xw$ and $yw$ are outer edges.
 Indeed, if $xw \in E_\text{in}(H)$, then by~\ref{enum:inner-edges-disconnect} the graph $H - \{x,w\}$ and therefore also the graph $G - \{x,w\}$ is disconnected.
 Let $K$ be a connected component of $G - \{x,w\}$ which does not contain $y$.
 By~\ref{enum:neighbors-in-components} we have $N(x) \cap V(K) \neq \emptyset$, as $G$ is $2$-connected.
 This is a contradiction to $N(x) - y = N(y) - x$.
 Thus for every $w \in S$ we have $xw \in E_\text{out}(H)$ and symmetrically $yw \in E_\text{out}(H)$.
 It follows from~\ref{enum:degree-2-vertices} that $\deg_H(w) = 2$ and hence $\deg_G(w) = 2$.
 Thus $V(G) = S \cup \{x,y\}$, as desired.

\subsection{Special decomposition of tree-width 2 graphs}

\begin{theorem}\label{thm:tw2-2connected}
 Let $G = (V,E)$ be a connected non-empty graph of tree-width at most~$2$, different from $C_4$.
 Then there exists a coloring $c: V \to \{1,2,3\}$ such that each of the following holds:
 \begin{enumerate}[label = \textbf{(\arabic*)}]
 \item For each $i \in \{1,2,3\}$ the set $V_i = \{v \in V \mid c(v) \neq i\}$ induces a forest $F_i$ in $G$.\label{enum:induce-forest}
 \item For each $i \in \{1,2,3\}$ there is no $K_1$-component in $F_i$.\label{enum:K1-components}
 \item For each $i \in \{1,2,3\}$ every $K_2$-component of $F_i$ is a twin edge.\label{enum:K2-components}
 \end{enumerate}
\end{theorem}
\begin{proof}
 We call a coloring $c: V \to \{1,2,3\}$ \emph{good} if it satisfies~\ref{enum:induce-forest}--\ref{enum:K2-components}
 We shall prove the existence of a good coloring by induction on $|V|$, the number of vertices in $G$. 
 We distinguish the cases whether $G$ is $2$-connected or not.
 
 \paragraph{Case~1: $G$ is not $2$-connected.}
 If $G$ is a single edge $uv$, then a desired coloring is given by $c(u) = c(v) = 1$.
 Otherwise $G$ has at least two blocks.
 Consider a leaf block $B$ in the block-cutvertex-tree of $G$ (see~\cite{West} for a definition) and the unique cut vertex $v$ of $G$ in this block.
 Consider the graphs $G_1 = B$ and $G_2 = G - (B - v)$, see Figure~\ref{fig:cut-vertex}.
 We define colorings $c_1$ and $c_2$ for $G_1$ and $G_2$, respectively, as follows.
 For $i \in \{1,2\}$, if $G_i \neq C_4$, then we apply induction to $G_i$ and obtain a coloring $c_i$ of $G_i$ satisfying~\ref{enum:induce-forest}--\ref{enum:K2-components}.
 On the other hand, if $G_i = C_4$, then we take the coloring $c_i$ shown in the left of Figure~\ref{fig:small-colorings}, in which the cut vertex $v$ is incident to the only $K_2$-component.
 Note that this coloring satisfies~\ref{enum:induce-forest} and~\ref{enum:K1-components}.

 \begin{figure}[htb]
 \centering
 \includegraphics{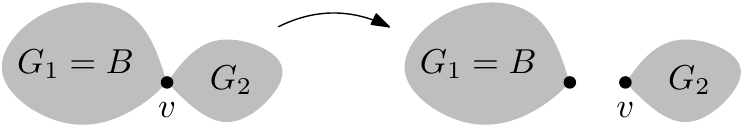}
 \caption{Splitting at cut vertex $v$.}
 \label{fig:cut-vertex}
 \end{figure}
 
 Without loss of generality (by permuting the colors if necessary) we have $c_1(v) = c_2(v) = 1$ and hence $c_1$ and $c_2$ can be combined into a coloring $c$ of $G$ by setting $c(x) = c_i(x)$ whenever $x \in V(G_i)$, $i=1,2$.
 Clearly, this coloring $c$ satisfies~\ref{enum:induce-forest} and~\ref{enum:K1-components}.
 
 If $xy$ is a $K_2$-component of $F_i$ in $G$ for some $i\in \{1,2,3\}$, with $v \neq x,y$, then $xy$ is also a $K_2$-component of the corresponding forest in $G_1$ or $G_2$, say in $G_1$.
 In particular, $G_1 \neq C_4$, since $v \neq x$, $y$.
 So $c_1$ satisfies~\ref{enum:K2-components} and $xy$ is a twin edge in $G_1$ and thus also in $G$, as desired.
 On the other hand, if $xy$ is an edge of $F_i$ in $G$ for some $i \in \{1,2,3\}$, say with $x = v$ and $y \in V(G_1)$, then $v$ is incident to another edge of $F_i$ in $G_2$, since $c_2$ satisfies~\ref{enum:K1-components}.
 Thus $xy$ is not a $K_2$-component of $F_i$.

 In any case, $c$ satisfies~\ref{enum:K2-components} and hence $c$ is a good coloring.

 \paragraph{Case~2: $G$ is $2$-connected.}

 \begin{figure}[htb]
 \centering
 \includegraphics{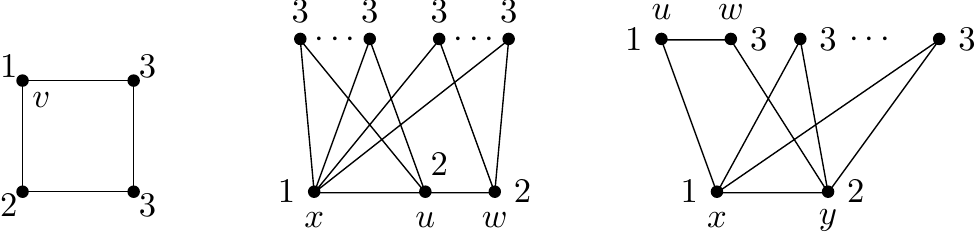}
 \caption{Left: A coloring of $C_4$ satisfying~\ref{enum:induce-forest} and~\ref{enum:K1-components} and with one $K_2$-component (in $F_3$). Middle and right: Good colorings when the graph obtained by contracting edge $uw$ has a twin edge.}
 \label{fig:small-colorings}
 \end{figure}
 
 Fix a $2$-tree $H$ with $G \subseteq H$ and $V(H) = V(G)$.
 Recall that by~\ref{enum:outer-are-in-G} we have $E_\text{out}(H) \subseteq E(G)$, i.e., every outer edge is an edge of $G$.
 An outer edge $e$ is called \emph{contractible} if $e$ is in no triangle of $G$.
 If $G$ has a contractible edge $e$, then we shall consider the smaller graph $G / e$ obtained by contracting $e$.
 If $G / e$ has a twin edge, we shall give a good coloring $c$ of $G$ directly, otherwise we obtain a good coloring $c$ by induction.
 On the other hand, if $G$ has no contractible edges, we shall give a good coloring $c$ directly.
 
 \paragraph{Case~2A: There exists a contractible edge $e$ in $G$.}
 Let $e = uw$ be contractible.
 Consider the graphs $G' = G / e$ and $H' = H / e$ obtained from $G$ and $H$ by contracting edge $e$ into a single vertex $v$.
 As $e \in E_\text{out}(H)$ and $H \neq K_3$ (otherwise $e$ would be in a triangle of $G$), we have that $H'$ is a $2$-tree.
 In particular, $\tw(G') \leq 2$.
 Moreover, by~\ref{enum:contracting-outer-is-good} $G'$ is also $2$-connected.
 Finally, as $e$ is not in a triangle in $G$, we have that $|E(G)|=|E(G') \cup \{e\}|$.
 
 If $G' = C_4$, then $G = C_5$ and it is easy to check that coloring the vertices around the cycle by $1,1,2,2,3$ gives a coloring $c$ satisfying~\ref{enum:induce-forest},~\ref{enum:K1-components} and~\ref{enum:K2-components}.

 If $G'$ has a twin edge $xy$, then by~\ref{enum:twins} we have that $G' = H'$ and $G'$ consists of $r$ triangles, $r \geq 1$, all sharing the common edge $xy$.
 Since the contractible edge $e$ lies in no triangle of $G$ and $G \neq C_4$, we have that $G'$ is not a triangle and thus in fact $r \geq 2$.
 Let $S = V(G') - \{x,y\}$.
 
 Now if $v = y$ (the case $v = x$ being symmetric), then $G$ looks like in Figure~\ref{fig:small-colorings} (middle) and a good coloring of $G$ is given by $c(x) = 1$, $c(u) = c(w) = 2$ and $c(z) = 3$ for every $z \in S$.
 On the other hand, if $v \in S$, then without loss of generality $ux \in E(G)$ and $wy \in E(G)$, and $G$ looks like in Figure~\ref{fig:small-colorings} (right side).
 A good coloring of $G$ is given by $c(x) = c(u) = 1$, $c(y) = 2$, and $c(z) = 3$ for every $z \in (S \cup w) - v$.
 In both cases it is easy to check that $c$ satisfies~\ref{enum:induce-forest},~\ref{enum:K1-components} and~\ref{enum:K2-components}.

 \begin{figure}[htb]
 \centering
 \includegraphics{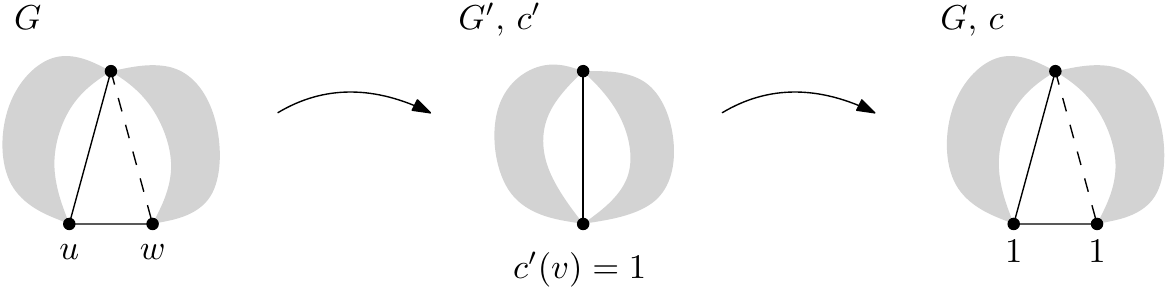}
 \caption{An example of the case that $G$ has a contractible edge $e = uw$. Grey areas indicate the remainder of the graph.}
 \label{fig:outerplanar-case-1}
 \end{figure}
 
 So finally we may assume that $G' \neq C_4$ and $G'$ has no twin edge.
 Applying induction to $G'$, we obtain a good coloring $c'$ of $V(G') = V - \{u,w\} + v$ with corresponding induced forests $F'_1$, $F'_2$ and $F'_3$ in $G'$.
 We define a coloring $c: V \to \{1,2,3\}$ by $c(x) = c'(x)$ for each $x \in V' - v$ and $c(u) = c(w) = c'(v)$, see Figure~\ref{fig:outerplanar-case-1}.
 
 Say $c'(v) = 1$.
 Now $c$ satisfies~\ref{enum:induce-forest} as $F_1 = F'_1$, and $F'_2$, $F'_3$ are obtained from $F_2$, $F_3$, respectively, by contracting the edge $uw$, that is not in a triangle in $G$.
 Moreover, $F'_i$ has no $K_1$ or $K_2$-components as $G'$ has no twin edge and~\ref{enum:K1-components},\ref{enum:K2-components} holds for $G'$.
 Thus $F_i$ has no $K_1$ or $K_2$-components, which shows that $c$ is a good coloring.

 \paragraph{Case~2B: There are no contractible edges in $G$.}
 In this case we define the coloring $c:V \to \{1,2,3\}$ to be 
 some proper $3$-coloring of $H$.
 (As mentioned in Section~\ref{sec:preliminaries}, such can be easily found through the construction sequence of $H$.)
 To prove that $c$ satisfies~\ref{enum:induce-forest}, assume for the sake of contradiction that there is a cycle in $F_i$ for some $i \in \{1,2,3\}$, i.e., a cycle using the two colors in $\{1,2,3\} - \{i\}$.
 Since $H$ is chordal, a shortest such cycle would be a $2$-colored triangle, which contradicts $c$ being a proper coloring.
 
 Next we prove that the coloring $c$ satisfies~\ref{enum:K1-components} and~\ref{enum:K2-components}.
 For any vertex $x \in V$, a \emph{trail around $x$} is defined to be a sequence $s_1,\ldots,s_r$ of $r$ distinct neighbors of $x$ in $H$, $r \geq 2$, such that $x,s_i,s_{i+1}$ form a triangle in $H$ for $i=1,\ldots,r-1$ and $xs_r \in E_\text{out}(H)$.
 Note that $xs_i \in E_\text{in}(H)$ for $i=2,\ldots,r-2$.
 Moreover, for any triangle $x,y,z$ in $H$ there is a trail around $x$ whose first elements are $s_1 = y$ and $s_2 = z$.
 Indeed, let $s_1 = y,s_2 = z,s_3,\ldots,s_r$ be an inclusion-maximal sequence of distinct neighbors of $x$ for which $x,s_i,s_{i+1}$ form a triangle in $H$ for $i=1,\ldots,r-1$.
 If $xs_r \in E_\text{out}(H)$, this is a trail as desired.
 Otherwise $xs_r \in E_\text{in}(H)$ and $x$, $s_r$ have another common neighbor $s_{r+1}$ in $H$ different from $s_{r-1}$.
 Moreover, $s_{r+1} \neq s_j$ for $j = 1,\ldots,r-2$, as otherwise the subgraph of $H$ induced by $\{x,s_j,\ldots,s_r\}$ would contain a subdivision of $K_4$, a contradiction to $\tw(H) = 2$.
 See Figure~\ref{fig:trails}.
 
 \begin{figure}[htb]
 \centering
 \includegraphics{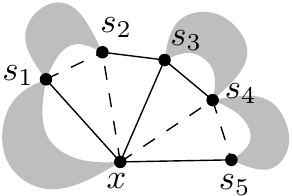}
 \caption{An example of a trail $s_1,\ldots,s_5$ around $x$.}
 \label{fig:trails}
 \end{figure}

 Now we shall show that $c$ satisfies~\ref{enum:K1-components} by proving that any vertex $x$ of $G$, say $c(x)=1$, has a neighbor in $G$ of color~$2$ and a neighbor in $G$ of color~$3$.
 As $G$ is connected, $x$ is adjacent to some vertex $y$, say $c(y) = 2$.
 The edge $xy$ is in a triangle in $H$ and its third vertex $z$ has color~$3$.
 Consider a trail $s_1,\ldots,s_r$ around $x$ starting with $s_1 = y$, $s_2 = z$.
 Note that $c(s_i) = 2$ if $i$ is odd and $c(s_i) = 3$ if $i$ is even.
 Hence, if $r$ is even, then as $xs_r \in E_\text{out}(H) \subseteq E(G)$, we have that $s_r$ is a neighbor of $x$ of color~$3$, as desired.
 
 Otherwise $r$ is odd, $r \geq 3$, and $xs_{r-1} \notin E(G)$.
 In particular $xs_r \in E_\text{out}(H)$ is in only one triangle of $H$, namely $x,s_{r-1},s_r$, and in no triangle of $G$.
 Hence $xs_r$ is contractible, contradicting the assumptions of \textbf{Case~2B}.
 This shows that $c$ satisfies~\ref{enum:K1-components}.
 
 Finally, to show that $c$ satisfies~\ref{enum:K2-components}, consider any edge $xy$ of $G$, say $c(x) = 1$ and $c(y) = 2$.
 If every trail around $x$ starting with $s_1 = y$ and every trail around $y$ starting with $s_1 = x$ has length $r=2$, then $xy$ is a twin edge.
 Otherwise, consider a longer such trail, say $s_1,\ldots,s_r$ is a trail around $x$ with $s_1 = y$ and $r \geq 3$.
 As before, note that $xs_r \in E_\text{out}(H)$ is an edge in $G$ and $c(s_i) = 2$ if $i$ is odd and $c(s_i) = 3$ if $i$ is even.
 If $xs_i \in E(G)$ for some odd $i \geq 3$, we are done.
 Otherwise $r$ is even, and $xs_{r-1} \notin E(G)$.
 As before, it follows that $xs_r$ is contractible, contradicting the assumptions of \textbf{Case~2B}.
 Hence $c$ also satisfies~\ref{enum:K2-components}, which completes the proof.
\end{proof}

 \subsection{Proof of Theorem~\ref{thm:tw3}}
 Let $G$ have tree-width $t$, then $G \subseteq H$ for some $t$-tree $H$.
 Then, as mentioned in Section~\ref{sec:preliminaries}, $\chi(H) =t+1$. 
 Consider a proper coloring of $H$ and assume that there is a cycle using two colors. Let $C$ be the shortest such cycle.
 Since $H$ is chordal, $C$ is a triangle. This is impossible since there are no $2$-colored triangles in a proper coloring.
Thus $ \chi_{\rm acyc}(H) = t+1$ and therefore $\chi_{\rm acyc}(G) \leq t+1$.
Theorem~\ref{thm:monotonicity}(i) immediately implies that $f_1(G) \leq \binom{t+1}{2}$.

\medskip

Next we shall consider $f_2(G)$, where $G$ is a graph of tree-width $t$. If $t=2$ and $G=C_4$, we see that each edge in $G$ is $2$-valid and two edge disjoint paths on $2$ edges each form two induced forests covering all the edges, so $f_2(C_4)=2$. If $t=2$ and $G\neq C_4$ is connected, then $f_2(G)\leq 3 = 3\binom{t+1}{3}$ by Theorem~\ref{thm:tw2-2connected}.
If $t=2$ and $G$ is not connected, then each component $G'$ of $G$ has tree-width at most $2$ and thus satisfies $f_2(G')\leq 3$ as argued above.
Picking one $2$-strong forest from each component of $G$ and taking their union yields a $2$-strong forest of $G$ and hence $f_2(G)\leq 3$.
 
Now, let $t\geq 3$. Given a graph $G$ of tree-width $t\geq 3$, let $H$ be a $t$-tree that contains $G$.
 It is well-known~\cite{DDO04}, that any proper $(t+1)$-coloring of $H$ has the property that any set of $p+1$ colors, $p = 1,\ldots,t$, induces a $p$-tree.
 We hence have a $(t+1)$-coloring of $G$ such that each of the $x= \binom{t+1}{3}$ sets of $3$ colors induces a graph of tree-width at most $2$.
 Call these graphs $G_1, \ldots, G_{x}$. As each $2$-valid edge has a witness tree induced by $3$ vertices, each witness tree is contained in $G_i$, for some $i\in\{1, \ldots, x\}$. 
 So each $2$-valid edge is $2$-valid in some $G_i$. 
 Since $\tw(G_i) \leq 2$, $f_2(G_i)\leq 3$, and so the $2$-valid edges of $G_i$ can be covered by $3$ $2$-strong forests, $i=1, \ldots, x$.
 Hence the $2$-valid edges of all $G_i$'s and thus the $2$-valid edges of $G$ can be covered with $3x = 3 \binom{t+1}{3}$ $2$-strong forests.
 \qed


\section{Minor-closed classes of graphs with bounded acyclic chromatic number}\label{sec:minor-closed}

\begin{lemma}\label{uncontract}
 Let $F$ be a graph and let $M$ be a matching in $F$.
 Let $F_M$ be the graph obtained by contracting the edges of $M$.
 \begin{itemize}
 \item If $F_M$ is a forest, then $\tw(F)\leq 3$. Moreover, if $M$ is an induced matching, then $\tw(F)\leq 2$. 
 
 \item Let $c$ be an acyclic coloring of $F_M$ with colors $1, \ldots, m$. 
 If $e$ is a $2$-valid edge of $F$ contained in $M$, then $e$ is $2$-valid in some subgraph $F_{a,b}$ of $F$, where $F_{a,b}$ is obtained by ``uncontracting" the subgraph of $F_M$ induced by colors $a$ and $b$, $a, b \in \{1, \ldots, m\}$.
 \end{itemize}
\end{lemma}
\begin{proof}
 To prove the first item assume that $F_M$ is a forest and let $S \subseteq V(F_M)$ be the set of vertices each corresponding to one contracted edge from $M$.
 Let $T_M$ be a tree that contains $F_M$ as an induced subgraph.
 In particular we have $\tw(F_M) \leq \tw(T_M) = 1$.
 Let $(T,B)$ be a tree decomposition of $T_M$ of width~$1$, that is, $T$ is a tree and for each vertex $v$ of $T_M$ there is a set $B_v \in B$, $B_v \subseteq V(T)$, such that $B_v$ induces a subtree of $T$ and $B_u \cap B_v \neq \emptyset$ if and only if $uv$ is an edge of $T_M$.
 Note that the ''only if`` follows as $T_M$ is a tree.
 
 We define a tree decomposition $(T',B')$ of $F$ as follows.
 Set $T' = T$ and $B'_v = B_v$ for every vertex $v \in V(F) - V(M)$.
 For an edge $uv \in M$ with corresponding vertex $w \in S$ we set $B'_u = B'_v = B_w$.
 Finally, set $B' = \{B'_v \mid v \in V(F)\}$.
 As $(T,B)$ has width~$1$, for any vertex $x$ in the tree $T$ we have $|\{v \in V(F_M) \mid x \in B_v\}| \leq 2$.
 As $M$ is a matching, every vertex in $F_M$ corresponds to at most two vertices in $F$ and we have $|\{v \in V(F) \mid x \in B'_v\}|\leq 4$.
 Moreover, if $M$ is an induced matching, then every pair of adjacent vertices in $F_M$ (and hence also $T_M$) corresponds to at most three vertices in $F$.
 This implies that $|\{v \in V(F) \mid x \in B'_v\}|\leq 3$, as for any $u,v \in V(F)$ we have $B_u \cap B_v \neq \emptyset$ only if $uv$ is an edge of $F_M$.

 \medskip
 
 To see the second item of the lemma, consider a witness tree of $e=xy$ with vertices $x,y,z$. Then $x$ and $y$ got contracted to a vertex of color, say $a$, in $F_M$ and $z$ either got contracted or stayed as it is and got some color $b$ under coloring $c$ of $F_M$. 
 So, $x, y, z \in V(F_{a,b})$. 
 Since $x,y,$ and $z$ induce a tree in $F$, they induce a tree in $F_{a,b}$ since $F_{a,b}$ is an induced subgraph of $F$.
\end{proof}

\subsection{Proof of Theorem~\ref{thm:acyclic}}
 
Given a graph $G\in\cC$, consider an acyclic coloring $c$ of $G$ with $x$ colors.
 For each of the $\binom{x}{2}$ many pairs of colors $\{i,j\}$, $i\neq j$, we split the forest induced by these colors into the $2$-strong forest $F_{i,j}$, and the induced matching $M_{i,j}$, which respectively gather the components with at least two edges and the ones with only one edge. Each edge of $G$ belongs to either $F_{i,j}$ or $M_{i,j}$ for some $i$, $j$. 
 Let $E$ be the set of edges that do not belong to any of $F_{i,j}$'s. Thus each $e\in E$ is in $M_{i,j}$, for some $i,j$. We see that the $\binom{x}{2}$ $2$-strong forests $F_{i,j}$ cover all $2$-valid edges of $G$ that are not in $E$.
 Next, we shall show two different approaches how to cover the $2$-valid edges of $G$ that are in $E$.

 Consider fixed $i$, $j$, $1\leq i<j\leq x$, and let $M=M_{i,j}$. 
 Let $G_M$ be the graph obtained from contracting the edges of $M$ in $G$.
 Then $G_M$ is again in the class $\cC$ and thus has acyclic chromatic number at most $x$. 
 Consider an acyclic coloring $c'$ of $G_M$ and the graph $H_{a,b}$ induced by two distinct color classes $a$ and $b$ in $G_M$.
 Consider $G_{a,b} = G_{a,b}(M)$, the graph obtained from $H_{a,b}$ by uncontracting $M$.
 Then, since $H_{a,b}$ is an induced subgraph of $G_M$, $G_{a,b}$ is an induced subgraph of $G$ and $H_{a,b}$ is obtained from $G_{a,b}$ by contracting the edges of $M$ in $G_{a,b}$.
 Thus, since $H_{a,b}$ is a forest and $M\cap E(G_{a,b})$ is an induced matching in $G_{a,b}$, by Lemma~\ref{uncontract} applied with $F=G_{a,b}$, we have $\tw(G_{a,b}) \leq 2$. Thus by Theorem~\ref{thm:tw3}, the $2$-valid edges of $G_{a,b}$ are covered by three $2$-strong forests. 
 By the second item of Lemma~\ref{uncontract}, applied with $F=G$, each $2$-valid edge of $G$ that is in $M$ is $2$-valid in some $G_{a,b}$.
 Each $2$-valid edge of $G$ from $E$ belongs to some matching $M= M_{i,j}$ and thus is $2$-valid in some $G_{a,b}(M)$. 
There are altogether $\binom{x}{2}$ such $M$'s and for each at most $\binom{x}{2}$ graphs $G_{a,b}(M)$, each contributing three covering forests. 
We see that all $2$-valid edges of $G$ from $E$ are covered by at most $3 \binom{x}{2} \binom{x}{2}$ $2$-strong forests in $G$.

 To see another way to deal with the edges in $E$ consider the subgraph $G'$ of $G$ formed by these edges.
 Vizing's theorem states (see e.g.~\cite{West}) that the edge set of any graph of maximum degree $D$ can be decomposed into at most $D+1$ matchings.
 Since each vertex of color $i$ under coloring $c$ is incident 
 to at most one vertex in each $M_{i,j}$, $ 1\leq j\leq x$, $j\neq i$, the maximum degree of $G'$ is at most $x-1$.
 Therefore, by Vizing's theorem, the edge set of $G'$ can be decomposed into at most $x$ matchings.
 Let $M$ be one such a matching. Let $G_M$ be obtained from $G$ by contracting $M$. Again, $G_M\in\cC$. Let $c'$ be an acyclic coloring of $G_M$ with at most $x$ colors and let $H_{a,b}$ be the induced forest formed by some color classes $a$ and $b$. Further, let $G_{a,b}$ be a graph obtained by uncontracting $M$ in $H_{a,b}$. 
 By Lemma~\ref{uncontract} applied with $F=G_{a,b}$, $\tw(G_{a,b})\leq 3$. Thus by Theorem~\ref{thm:tw3}, the $2$-valid edges of $G_{a,b}$, can be covered by twelve $2$-strong forests. By the second item of Lemma~\ref{uncontract} applied with $F=G$, each $2$-valid edge of $G$ that is in $M$ is $2$-valid in some $G_{a,b}$. Therefore, all $2$-valid edges of $G$ from $E$ are covered by at most $12x \binom{x}{2}$ $2$-strong forests.

 So, the $2$-valid edges from $E$ are covered by at most $\min\{ 12x \binom{x}{2}, \binom{x}{2}3\binom{x}{2}\}$  $2$-strong forests.
 Recall that the remaining $2$-valid edges that are in $F_{i,j}$s are covered by at most $\binom{x}{2}$ $2$-strong forests. 
 Finally, using Borodin's result that each planar graph has acyclic chromatic number at most~$x = 5$~\cite{Bor79}, the first bound implies that for every planar graph $G$ we have $f_2(G)\leq \binom{5}{2}(3\binom{5}{2}+1) = 310$.
 The theorem follows. 
 \qed

\section{Graphs of bounded tree-depth}\label{sec:tree-depth}

Recall that $G$ has tree-depth at most $d$ if and only if there exists a rooted forest $F$ of depth $d$ such that $G$ is a subgraph of the closure of $F$.
When $F$ consists of only one tree and $V(G)=V(F)$ we call such a tree an \emph{underlying tree} of $G$.
In particular any connected graph of tree-depth at most $d$ has an underlying tree of depth at most $d$.
Note that also disconnected graphs can have such an underlying tree.
For example when $G$ admits an underlying tree of depth $d$, then any graph obtained from $G$ by adding isolated vertices also admits an underlying tree of depth $d$.
Let $\TD(d)$ denote the set of all graphs of tree-depth at most $d$ having an underlying tree of depth $d$ and let $\TD^\star(d) \subseteq \TD(d)$ be the set of all graphs $G$ in $\TD(d)$ some of whose underlying trees of depth $d$ have the root of degree $1$.
When we talk about a graph $G$ of tree-depth at most $d$, we usually associate a fixed underlying tree $T$ with root $r$ to it.
Let $f_k(d) = \sup \{f_k(G)\mid \td(G) \leq d\}$.
We shall inductively show that this function is bounded by $(2k)^d$ from above and hence $\sup$ could be replaced by $\max$ in the definition of $f_k(d)$.


\begin{lemma}\label{lem:cut-depth}
 Let $G\in \TD^{\star}(d)$ with underlying tree $T$ of depth at most $d$ having a root $r$ of degree $1$ that is adjacent to a vertex $x$ in $T$.
 Then $G-r, G-x \in \TD(d-1)$.
\end{lemma}
\begin{proof}
 It suffices to observe that $G-r$, $G-x$ are graphs of tree-depth at most $d-1$ with underlying trees obtained 
 by removing $r$ from $T$, or removing $r$ in $T$ and renaming $x$ with $r$, respectively.
The roots of these trees are $x$ and $r$, respectively.
\end{proof}

An edge $e$ is \emph{almost $k$-valid} in a graph $G \in \TD(d)$ with associated root $r$ if it is not $k$-valid in $G$ but there is an induced path in $G$ containing $r$ and $e$, i.e., both endpoints of $e$.
Note that, for example, if $e = xy$ and $xr, yr \in E(G)$, then there is no such induced path.
Let $g_k(d)$ and $g^\star_k(d)$ be the the maximum number of almost $k$-valid edges in a graph $G\in\TD(d)$, respectively $G\in\TD^\star(d)$.

\begin{lemma}\label{lem:rootPaths}
 For all positive integers $k$, $d$, with $d\geq 2$, we have $g_k(d) \leq (2k)^{d-1}-1$ and $g^\star_k(d)\leq 2 (2k)^{d-2}-1$.
\end{lemma}
\begin{proof}

 For a fixed $k$ we prove the claim by induction on $d$.
 If $d=2$, then any $G\in\TD(d)$ is a subgraph of a star.
 Therefore either all edges form a $k$-strong forest or $G$ has at most $k-1$ edges and thus each edge is almost $k$-valid.
 Hence $g_k(2)=k-1$ for any $k \geq 1$, $g^\star_1(2)=0$, and $g^\star_k(2)=1$ for $k\geq 2$.
 
 \medskip
 
 Now suppose that $d\geq 3$ and that the statement of the lemma is true for smaller values of $d$.
 We consider $g^\star_k(d)$ first.
 Let $G\in\TD^\star(d)$, $r$ be the root of the underlying tree $T$ of $G$, and $x$ be the unique neighbor of $r$ in $T$.
 Let $A$ be the set of almost $k$-valid edges $e$ in $G$ such that there is an induced path in $G$ containing $e$, $r$, and not containing $x$.
 Let $B$ be the set of all remaining almost $k$-valid edges in $G$. 
 Each edge in $A$ is almost $k$-valid in $G-x$.
 Similarly each edge in $B\setminus\{rx\}$ is almost $k$-valid in $G-r$ (here the underlying tree is as in Lemma~\ref{lem:cut-depth}). Note that $rx$ might or might not be an edge of $G$ and if it is an edge, it is $k$-valid or almost $k$-valid.
 Since $G-r, G-x \in \TD(d-1)$ by Lemma~\ref{lem:cut-depth}, we conclude that $|A|, |B\setminus\{rx\}|\leq g_k(d-1)$.
 Inductively we obtain $|A|+|B| \leq 2\cdot ((2k)^{d-2}-1) + 1 = 2\cdot (2k)^{d-2}-1$.
 Since $G\in\TD^\star(d)$ was arbitrary we have that $g^\star_k(d) \leq 2\cdot (2k)^{d-2}-1$.
 
 \medskip
 
 Now consider $g_k(d)$ for $d\geq 3$.
 Let $G\in\TD(d)$ and let $r$ be the root of the underlying tree $T$ of $G$.
 Let $x_1, \ldots, x_t$ be the neighbors of $r$ in $T$ and let $G_i$, $1\leq i\leq t$, be the subgraph of $G$ induced by $i^{\rm th}$ branch of $T$, i.e., by
 $r$, $x_i$, and all descendants of $x_i$ in $T$.
 Assume that each of $G_1, \ldots, G_s$ has an edge incident to $r$ and the other $G_i$'s do not have such an edge. Then, in particular, each almost $k$-valid edge of $G$ is in some $G_i$, $i=1, \ldots, s$.

 Assume first that $s\geq k$. There is a star $S$ of size $s$ with center $r$ and edges being from distinct $G_i$'s. 
 If $e$ is an edge contained in some induced path $P$ in $G_i$ where $r$ is an endpoint, then there is an induced tree in $G$ formed by $P$ 
 and all edges of $S$ except perhaps for the edge from $G_i$.
 So any such edge is $k$-valid and there are no almost $k$-valid edges in $G$.
 
 Now assume that $s\leq k-1$.
 Any almost $k$-valid edge $e \in E(G)$ is almost $k$-valid in $G_i$, for some $i\in[s]$.
 There are at most $g_k^\star(d)$ almost $k$-valid edges in $E(G_i)$, $i=1,\ldots,s$. 
 Therefore the total number of almost $k$-valid edges in $G$ is at most $s\cdot g^\star_k(d) \leq (k-1)\, g^\star_k(d)\leq (k-1)\, (2^{d-1} k^{d-2}-1)\leq 2^{d-1} k^{d-1}-1$.
 
 Since $G\in\TD(d)$ was arbitrary we have that $g_k(d) \leq (2k)^{d-1}-1$.
\end{proof}

\subsection{Proof of Theorem~\ref{thm:tree-depth}}

Let $G$ be a graph of tree-depth $d$. First of all consider the case $k=1$. 
 It is well-known (see~\cite{BGH95,NO08}) that any graph of tree-depth at most $d$ has tree-width at most $d-1$.
 Hence if $\td(G) \leq d$, then by Theorem~\ref{thm:tw3} we have $f_1(G) \leq \binom{d}{2}$.
 On the other hand $K_d$ is of tree-depth $d$ and $f_1(K_d) = \binom{d}{2}$, so the above bound is tight.

 \medskip
 
 For the rest of the proof assume that $k \geq 2$. 
 We prove that $f_k(G) \leq (2k)^d$ for any graph of tree-depth $d$.
 First we prove this claim for $G\in\TD(d)$ by induction on $d$, then we deduce the general case.
 Recall that $G\in\TD(d)$ if and only if $G$ has an underlying tree of depth $d$.
 
 If $d=1$, then any graph in $\TD(d)$ has no edges.
 If $d=2$, then any $G\in\TD(d)$ is a subgraph of a star. If $G$ has at least $k$ edges, then $G$ is a $k$-strong forest itself. If $G$ has less than $k$ edges, there are no $k$-valid edges.
 Hence $f_k(G)\leq 1$.
 
 Now suppose that $d \geq 3$ and assume that $f_k(G)\leq (2k)^{d'}$ for any $G\in\TD(d')$ and $d'<d$.
Let $r$ denote the root of the underlying tree $T$ of $G$.
 Let $x_1, \ldots, x_t$ be the neighbors of $r$ in $T$ and let $G_i$, $1\leq i\leq t$ be the subgraph of $G$ induced by $i^{\rm th}$ branch of $T$, i.e., by $r$, $x_i$, and all descendants of $x_i$ in $T$.
 Then $G_i\in\TD^\star(d)$, where in the corresponding underlying tree $r$ is the root and $x_i$ is its unique neighbor, $i=1,\ldots,t$. Here the underlying trees for subgraphs are defined as in Lemma~\ref{lem:cut-depth}.

 Let $E$ be the set of $k$-valid edges in $G$. We shall split $E$ into sets $S_1, \ldots, S_5$ and shall show that each of these sets is covered by a desired number of $k$-strong forests, see Figure~\ref{fig:treedepthPartition}.

\begin{itemize}
\item Let $S_1 = \{rx_i : ~ i=1,\ldots,t\}\cap E $.

\item
Let $S_2$ be the set of edges from $ E \setminus S_1$ that are $k$-valid in $G_i - r$ for some $i \in \{1,\ldots,t\}$.

\item
 Let $S_3 $ be the set of edges from $ E \setminus (S_1\cup S_2)$ that are $k$-valid in $G_i - x_i$ for some $i \in \{1,\ldots,t\}$.

\item
Let $S_4$ be the set of edges from $E \setminus (S_1 \cup S_2\cup S_3)$ that are $k$-valid in $G_i$ for some $i \in \{1,\ldots,t\}$.

\item
Let $S_5 =E - (S_1 \cup S_2 \cup S_3\cup S_4)$.
\end{itemize}

 \begin{figure}
 \centering
 \includegraphics{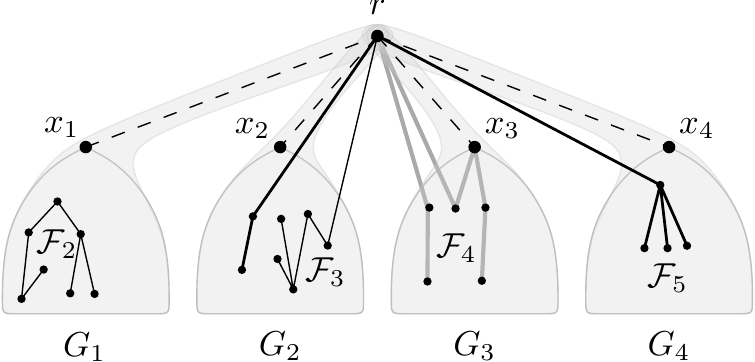}
 \caption{Illustration of the proof of Theorem~\ref{thm:tree-depth}.}
 \label{fig:treedepthPartition}
\end{figure}

I.e., $S_2, S_3$, and $S_4$ consist of $k$-valid edges in some $G_i$, with witness trees not containing $r$, not containing $x_i$, and containing both $r$ and $x_i$, respectively. 
Each edge in $S_5$ is not $k$-valid in any $G_i$, but it is almost $k$-valid in some $G_i$. 
In the following, we say that a family of forests is a \emph{good cover} of an edge set if these covering forests are $k$-strong.

 \begin{claim}
 There exists a good cover $\cF_1$ of $S_1$ of size at most $k-1$.
 \end{claim} 
 \begin{claimproof}
 If $|S_1| < k$, for each $e\in S_1$ pick a $k$-strong forest in $G$ containing $e$ and let $\cF_1$ be the set of all these forests.
 If $|S_1| \geq k$, then let $\cF_1$ consist of one forest that is the induced star with edge set $S_1$.
 \end{claimproof}

 \begin{claim}
 There exists a good cover $\cF_2$ of $S_2$ of size at most $f_k(d-1)$.
 \end{claim}
 \begin{claimproof}
 Let $i \in \{1,\ldots,t\}$.
 By Lemma~\ref{lem:cut-depth} we have that $G_i-r\in\TD(d-1)$.
 Hence we have $f_k(G_i-r)\leq f_k(d-1)$.
 Let $\cA_i$ denote a good cover of $S_2\cap E(G_i-r)$ of size at most $f_k(d-1)$ with forests contained in $G_i-r$.
 We shall combine the forests from $\cA_1,\ldots,\cA_t$ into a new family $\cF_2$ of at most $f_k(d-1)$ $k$-strong forests of $G$.
 Any union $F_1\cup \cdots \cup F_t$, where $F_i\in \cA_i$, is a $k$-strong forest in $G$ because none of these forests contain $r$ and thus there are no edges between $F_i$ and $F_j$ for $1\leq i <j\leq t$.
 So, let each forest from $\cF_2$ be a union of at most one forest from each $\cA_i$. We see that we can form such a family of size at most $\max \{|\cA_i|: 1\leq i\leq t\}$. Thus $\cF_2$ is a family of at most $f_k(d-1)$ $k$-strong forests of $G$.
 Since each edge $e\in S_2$ is $k$-valid in some $G_i-r$, the set $\cF_2$ is a good cover of $S_2$.
 \end{claimproof}

 \begin{claim}
 There exists a good cover $\cF_3$ of $S_3$ of size at most $f_k(d-1)$.
 \end{claim}
\begin{claimproof}
Let $i \in \{1,\ldots,t\}$. By Lemma~\ref{lem:cut-depth} we have that $G_i-x_i\in\TD(d-1)$.
 Hence we have $f_k(G_i-x_i) \leq f_k(d-1)$.
 Let $\cA_i$ denote a good cover of $S_3 \cap E(G_i - x_i)$ consisting of at most $f_k(d-1)$ forests in $G_i - x_i$.
 Similarly as in the claim before, any union $F_1\cup \cdots \cup F_t$, where $F_i\in \cA_i$, is a $k$-strong forest in $G$ because all of these forests contain $r$ as $S_2\cap S_3=\emptyset$.
 Let each forest in $\cF_3$ be a union of at most one forest from each of $\cA_i$, $i=1, \ldots, t$. 
 It is clear that one can build such a family with at most $\max \{|\cA_i| : 1\leq i\leq t\}$ forests. 
So, $\cF_3$ consists of at most $f_k(d-1)$ $k$-strong forests of $G$. Since each edge $e\in S_3$ is $k$-valid in some $G_i-x_i$, 
the set $\cF_3$ is a good cover of $S_3$.
 \end{claimproof}

%
%
Recall that $g_k(d)$ denotes the maximum number of almost $k$-valid edges in a graph $G \in \TD(d)$, where an edge $e$ is almost $k$-valid if it is not $k$-valid but there exists an induced path containing $e$ and the root $r$ of the underlying tree $T$.

 \begin{claim}
 There exists a good cover $\cF_4$ of $S_4$ of size at most $2 g_k(d-1)$.
 \end{claim}
 \begin{claimproof}
 Let $i \in \{1,\ldots,t\}$ and let $e \in S_4$. Then $e$ is $k$-valid in $G_i$. Since $e$ is not $k$-valid in $G_i-r$ and not $k$-valid in $G_i - x_i$, this means that each witness tree of $e$ in $G_i$ contains both $r$ and $x_i$. Every such witness tree contains a path containing $e$, $x_i$ and not $r$, or a path containing $e$, $r$ and not $x_i$.
 This path is induced, as witness trees are induced, and thus (as $e \notin S_1$) $e$ is almost $k$-valid in either $G_i-r$ or $G_i -x_i$, respectively. 
 Hence, $|S_4 \cap G_i| \leq 2 g_k(d-1)$ by the definition of $g_k(d-1)$, since $G_i-r, G_i-x_i \in \TD(d-1)$ by Lemma~\ref{lem:cut-depth}.

 For each edge $e$ in $S_4 \cap G_i$ we pick one witness tree of $e$ that is contained in $G_i$.
 Let $\cA_i$ denote the set of these at most $2g_k(d-1)$ $k$-strong forests.
 As all induced forests in $\cA_1,\ldots,\cA_t$ contain the root $r$, we can again, as in the previous claim, form a set $\cF_4$ of at most $2 g_k(d-1)$ $k$-strong forests in $G$ covering $S_4$.
 \end{claimproof}

%
 \begin{claim}
 There exists a good cover $\cF_5$ of $S_5$ of size at most $(k-1)\, g^\star_k(d)$.
 \end{claim}
 \begin{claimproof}
 Note that $S_5$ consists of those edges whose witness trees all contain edges from at least two different $G_i$'s.
 Without loss of generality assume that each of $G_1, \ldots, G_s$ have an edge incident to $r$ and the other $G_i$'s do not have such an edge.
 Then each $e\in S_5$ is in $G_i$ for some $i\in\{1,\ldots,s\}$ and moreover $e$ is almost $k$-valid in this $G_i$.
 Hence $|E(G_i)\cap S_5|\leq g^\star_k(d)$ for all $i$, $1\leq i\leq s$, and $|E(G_i)\cap S_5|=0$ for all $i$, $s< i\leq t$.

 If $s \leq k-1$, then $|S_5| \leq (k-1)\, g^\star_k(d)$.
 In this case we let each forest in $\cF_5$ consist of one witness tree for each $e\in S_5$.
 
 If $s \geq k$, then for all $i \in \{1,\ldots,s\}$ and all $j$, $1\leq j\leq g^\star_k(d)$, we pick (not necessarily distinct) induced paths $P_i^j$ in $G_i$ starting with $r$ such that each edge in $S_5$ is contained in some path $P_i^j$. Such a path containing $e \in S_5$ exists, since $e$ is almost $k$-valid in some $G_i$.
 As $s \geq k$, the union of paths $P_1^j \cup \cdots \cup P_s^j$ forms a $k$-strong forest in $G$ for each $j$, $1\leq j\leq g^\star_k(d)$.
 Moreover each edge in $S_5$ is contained in one of these forests.
 Hence $\cF_5 = \{P_1^j \cup \cdots \cup P_s^j \mid 1\leq j\leq g^\star_k(d)\}$ is a good cover of $S_5$ as desired.
 \end{claimproof}

 From the above claims we get that $\cF = \cF_1 \cup \cF_2 \cup \cF_3 \cup \cF_4 \cup \cF_5$ is a good cover of all $k$-valid edges in $G$, since $S_1 \cup S_2 \cup S_3 \cup S_4 \cup S_5$ contains all $k$-valid edges of $G$.
 Moreover, with the above claims, induction, Lemma~\ref{lem:rootPaths} and $k \geq 2$ we get
 \begin{align*}
 |\cF| &\leq |\cF_1| + |\cF_2| + |\cF_3| + |\cF_4| + |\cF_5|\leq k-1 + 2\,f_k(d-1)+2\,g_k(d-1) + (k-1)\, g^\star_k(d)\\
 &\leq k-1 + 2 \cdot (2k)^{d-1} + 2\cdot ((2k)^{d-2}-1) + (k-1) (2 (2 k)^{d-2}-1)\\
 &\leq (2k)^{d-2}( 2\cdot 2k + 2 + 2(k-1)) = (2k)^{d-2} 6k \leq (2k)^{d-2} 4k^2 = (2k)^d,
 \end{align*}
 which proves that $f_k(G) \leq (2k)^d$ for $G\in\TD(d)$.
 
 Now if $G$ has tree-depth at most $d$, then each component of $G$ is in $\TD(d)$.
 By the previous arguments all $k$-valid edges of such a component can be covered by at most $(2k)^d$ $k$-strong forests.
 A union of one such forest from each component is a $k$-strong forest in $G$.
 Hence we can form at most $(2k)^d$ $k$-strong forests of $G$ that cover all $k$-valid edges of $G$.
 Thus $f_k(G) \leq (2k)^d$ for each graph $G$ of tree-depth at most $d$.


Finally we shall prove that $f_k(G) \leq (2k)^{k+1} \binom{d}{k+1}$, for $d\geq k+1$ and a graph $G$ with $\td(G) \leq d$. Let $H$ be a maximal tree-depth $d$ supergraph of $G$ on the same set of vertices.
 It is known~\cite{NO06,NO08}, that there is a proper $d$-coloring of $H$ (a so called $d$-centered coloring) such that any set of $p$ colors, $p \leq d$, induces a tree-depth~$p$ graph.
 We hence have a $d$-coloring of $G$ such that each of the $\binom{d}{k+1}$ subsets of $(k+1)$ colors induces a graph of tree-depth $k+1$.
As each $k$-valid edge has a witness tree induced by $k+1$ vertices, each witness tree belongs to one of these $\binom{d}{k+1}$ graphs.
 So each $k$-valid edge of $G$ is $k$-valid in (at least) one of these graphs. Thus the total number of $k$-strong forests covering $k$-valid edges of $G$ is at most 
 $\binom{d}{k+1} (2k)^{k+1}$, where the bound $(2k)^{k+1}$ comes from the first part of the theorem when $d=k+1$. 
 \qed

 \section{Proof of Theorem~\ref{thm:main}}\label{sec:main}
Recall that $\chi_p(G)$ is the smallest integer $q$ such that $G$ admits a vertex coloring with $q$ colors such that for each $p' \leq p$ each $p'$-set of colors induces a subgraph of $G$ of tree-depth at most $p'$. 
Since $\cC$ is of bounded expansion there is a sequence of integers $b_1,b_2,\ldots$, such that for any graph $G \in \cC$ and any $p$, $\chi_p(G)\leq b_p$. 

Let $G\in \cC$.
Note that $\chi_{\rm acyc}(G) \leq \chi_2(G)$ and hence Theorem~\ref{thm:monotonicity}\ref{enumii:a-vs-f_1} gives $f_1(G) \leq \binom{\chi_2(G)}{2}$.
Thus we can take $c_1 = \binom{b_2}{2}$. For $k \geq 2$, consider a $(k+1)$-tree-depth coloring of $G$ with $\chi_{k+1}(G) \leq b_{k+1}$ colors.
For each of the $\binom{\chi_{k+1}(G)}{k+1}$ subgraphs $H$ induced by $k+1$ colors, consider a cover of the $k$-valid edges in $H$ with $f_k(H)$ $k$-strong forests.
Note that each $k$-valid edge of $G$ is $k$-valid in at least one of these graphs.
Indeed, a witness tree of an edge $e$ is induced by $k+1$ vertices, that are colored with at most $k+1$ different colors, hence $e$ is $k$-valid in a graph $H$ induced by these colors.
Then the union $\cF$ of all these forests is a cover of all $k$-valid edges in $G$.
Finally observe that each $H$ has tree-depth at most $k+1$, and thus we have $f_k(H)\leq (2k)^{k+1}$ by Theorem~\ref{thm:tree-depth}.
Hence $|\cF| \leq \binom{b_{k+1}}{k+1} (2k)^{k+1} $, and we can take $c_k= \binom{b_{k+1}}{k+1} (2k)^{k+1}$.

\section{Conclusions}\label{sec:Conclusions}

In this paper we introduce the $k$-strong induced arboricity $f_k(G)$ of a graph $G$ to be the smallest number of $k$-strong forests covering the $k$-valid edges of $G$, where a forest is $k$-strong if it is induced and all its components have size at least $k$, and an edge is $k$-valid if it belongs to some $k$-strong forest.
Recall that for $k \in \bN$, we call a class $\cC$ of graphs \emph{$f_k$-bounded}, if there is a constant $c=c(\cC,k)$ such that $f_k(G)\leq c$ for each $G\in \cC$, and that $\cC$ is \emph{$f$-bounded} if $\cC$ is $f_k$-bounded for each $k \in \bN$.

We show that this new graph parameter $f_k$ is non-monotone as a function of $k$ and, for any $k\geq 2$, as a function of $G$ using induced subgraph partial ordering. 
Indeed, there exist classes of graphs $\cC_k$ and $\cC'_k$, $k \geq 2$, such that $\cC_k$ is $f_k$-bounded but not $f_{k+1}$-bounded, while $\cC'_k$ is $f_{k+1}$-bounded but not $f_k$-bounded.
Nevertheless, $f_k$ is bounded for graph classes of bounded expansion, which in particular includes proper minor-closed families.
Our main result is that every class $\cC$ of bounded expansion is $f$-bounded.
This implies, in particular, that the adjacent closed vertex-distinguishing number for graphs from such classes is bounded by a constant, which includes for example all planar graphs.
Additionally,
we find upper and lower bounds on $f_1(G)$, the induced arboricity, and study the relation between $f_1$ and the well-known notions of arboricity and acyclic chromatic number.

It remains open to improve the lower and upper bounds on $f_k$ for a given graph class.
For example, for planar graphs the maximum value for $f_1$ is between $6$ (as certified by $K_4$) and $10$ (following from $f_1(G) \leq \binom{X_{\rm acyc}(G)}{2}$ and Borodin's result on the acyclic chromatic number of planar graphs~\cite{Bor79}).
For graphs $G$ of tree-width~$t$, we provide explicit universal upper bounds on $f_1(G)$ and $f_2(G)$ in Theorem~\ref{thm:tw3}.
For $k \geq 3$, Theorem~\ref{thm:main} states the existence of a constant upper bound.
Using $f_k(G) \leq \binom{\chi_{k+1}(G)}{k+1}(2k)^{k+1}$, provided $\chi_{k+1}(G) \geq k+1$, from the proof of Theorem~\ref{thm:main} and $\chi_{p}(G) \leq t^p+1$ for $G$ of tree-width~$t$~\cite{NO06}, we conclude that $f_k(G) \leq \binom{t^{k+1}+1}{k+1}(2k)^{k+1}$ for any integer $k$ and any graph $G$ of tree-width~$t \geq 2$.
This upper bound is most likely far from the actual value, and improving its order of magnitude seems to be an interesting challenge.

\paragraph{Variations of $k$-strong induced arboricity:}
Recall that in the definition of $f_k$ 
edges that are not $k$-valid need not be (in fact, can not be) covered at all.
In particular, when no edge in $G$ is $k$-valid, then $f_k(G) = 0$.
By handling edges that are not $k$-valid differently, one obtains alternative definitions of our new arboricity parameter, which we discuss here briefly.

\medskip

Firstly, one might want to insist on covering \emph{all} edges (not just all $k$-valid edges) of $G$ with $k$-strong forests and define the corresponding parameter $f'_k(G)$.
We would have $f'_k(G) = \infty$ whenever not all edges of $G$ are $k$-valid.
On the positive side, we see that $f'_{k+1}(G) \geq f'_k(G)$ for every graph $G$ and every $k \geq 1$, because every $(k+1)$-strong forest is also a $k$-strong forest.
But $f'_k$ is again not monotone using induced subgraph partial ordering, since $f'_k(K_2) = \infty$ for every $k \geq 2$.
This also shows that no graph class $\cC$ that contains $K_2$ is $f'_k$-bounded for $k \geq 2$.
Even worse, we have $f'_k(G) = \infty$ for every $k > |E(G)|$ and hence no graph class is $f'$-bounded.
However, for every graph $G$ we have $f'_1(G) = f_1(G)$ (since every edge is $1$-valid) and hence a class $\cC$ is $f'_1$-bounded if and only if $\cC$ is $f_1$-bounded.

\medskip

Another natural variation of the $k$-strong induced arboricity would be the following: 
For a set $S \subseteq \bN$ of natural numbers and a graph $G$ define $f_S(G)$ to be the minimum number of induced forests in $G$ such that for all $k \in S$ each $k$-valid edge in $G$ lies in a component of size at least $k$ of some of the forests.
Clearly, we have $f_k(G) = f_{\{k\}}(G)$ and for $T \subset S$ we have $f_T(G) \leq f_S(G)$.
In particular considering $S = [k] = \{1,\ldots,k\}$ gives a parameter $f_{[k]}(G)$ similar to the $k$-strong induced arboricity in which however every edge of $G$ has to be covered.
As before, we say that a graph class $\cC$ is $f_S$-bounded if there is a constant $c_S = c(\cC,S)$ such that $f_S(G) \leq c_S$ for all $G \in \cC$.
It follows from our results, that for any \emph{finite} set $S \subset \bN$ any class $\cC$ of bounded expansion is $f_S$-bounded, because $f_S(G) \leq \sum_{k \in S} f_k(G)$ for every set $S$ and graph $G$.
On the other hand, the examples in Theorem~\ref{thm:monotonicity} show that the class of tree-width~$2$ graphs is not $f_\bN$-bounded, even the class of tree-depth~$3$ graphs, and the graphs of maximum degree at most~$4$.
(Note that all these classes are of bounded expansion, c.f.\ Figure~\ref{fig:graph-class-properties}.)
It would however be interesting to identify non-trivial graph classes (maybe grids?) that are $f_\bN$-bounded.
For example, one can show that $f_\bN(G) \leq 5$, whenever $G$ is a wheel graph, c.f.\ Figure~\ref{fig:small-examples}.
Moreover, in Theorem~\ref{thm:monotonicity}\ref{enumi:f_k-notNowhereDense} we present a graph class $\cC$ that is not nowhere dense, for which $f_\bN(G) \leq 2$ for each $G \in \cC$.





\subsection*{Acknowledgments}

We would like to thank the anonymous reviewers for their careful reading and helpful comments.

\bibliographystyle{plain}
\bibliography{lit}

\end{document}